\documentclass[11pt,a4paper]{article}
\usepackage{amssymb,amsthm,thmtools}
\usepackage[tbtags]{amsmath}
\usepackage{tcolorbox}
\usepackage{pgfplots}
\usepackage[british]{babel}
\usepackage[utf8]{inputenc}
\usepackage{amsthm}
\usepackage{hyperref} 
\usepackage{cleveref} 
\usepackage{tabularx}
\usepackage{array}
\usepackage{enumitem}
\usepackage{wrapfig}
\numberwithin{equation}{section}
\usepackage{relsize}
\usepackage{mathtools}	
\usepackage{accents}
\usepackage{multirow}
\usepackage{subcaption}
\usepackage{graphicx}
\usepackage{amscd}
\usepackage{epic, eepic}
\usepackage{url}
\usepackage{xcolor}
\usepackage{todonotes}
\usepackage{enumitem}
\usepackage{comment}
\usepackage{hhline}
\usepackage{stringenc}
\usepackage{titling}
\addto\extrasbritish{

}

\usepackage{tikz}
\usepackage{ifthen}
\usepackage{mathrsfs}
\pgfplotsset{compat=1.15}
\usetikzlibrary{arrows}
\definecolor{ffffvc}{rgb}{1.,1.,0.3607843137254902}
\definecolor{wwzzff}{rgb}{0.4,0.6,1.}

\usepackage{tikz-cd}

\newcommand{\nn}{\mathbb{N}}
\newcommand{\zz}{\mathbb{Z}}

\newcommand{\Z}{\mathbb{Z}}

\declaretheorem[style=plain,name=Theorem,numberwithin=section]{theorem}

\declaretheorem[style=plain,name=Lemma,sibling=theorem]{lemma}
\declaretheorem[style=plain,name=Problem,sibling=theorem]{problem}

\declaretheorem[style=plain,name=Corollary,sibling=theorem]{corollary}
\declaretheorem[style=plain,name=Claim,sibling=theorem]{claim}

\declaretheorem[style=definition,name=Definition,sibling=theorem]{definition}

\declaretheorem[style=definition,name=Remark,sibling=theorem]{remark}
\declaretheorem[style=remark,name=Note,sibling=theorem]{note}

\newlist{proplist}{enumerate}{1}
\setlist[proplist]{
    label={(\thedefinition.\arabic*)},
    ref={\thedefinition.\arabic*},
    before={\cpagelogic}
}

\allowdisplaybreaks
\AtBeginDocument{
    \hypersetup{
      hidelinks,
      colorlinks=false,    
      pdfencoding=auto
    }
}
\date{}
\title{The extensible no-$(k(n)+1)$-in-line problem}
\author{
    Tamás Gábriel\thanks{
    Eötvös Loránd University.
    \texttt{gabriel5.tamas@gmail.com}.}
    \and
    Máté Jánosik\thanks{
    Eötvös Loránd University.
    \texttt{janosikmate6@gmail.com}.}
    \and
    Dávid Melján\thanks{
    Eötvös Loránd University.
    \texttt{david.meljan@gmail.com}.}
    \and
    Benedek Nádor\thanks{
    Eötvös Loránd University.
    \texttt{nador.benedek@gmail.com}.}
}

\begin{document}

\maketitle
\begin{abstract}
\noindent
The classical no-$k$-in-line problem asks for the largest number of points that can be placed on an $n \times n$ grid without having $k$ of them collinear. A natural extension, motivated by the analogous question by Erde for $k\in \mathbb{Z}$, is the \emph{extensible no-$(k(n)+1)$-in-line problem}, which seeks a subset of points in $\mathbb{Z}^2$ with maximal possible density such that at most $k(n)$ points are collinear within the subgrid $[1,n]^2$.

\noindent We construct optimal sets for linear functions and positive-density sets for power functions. We prove that any configuration achieving $\liminf\frac{S_n}{n k(n)} \ge 0.897$ must satisfy $k(n) = \Omega( n^c)$ for some $c>0$ constant; therefore, the extensible no-$k$-in-line problem has no configuration with this property when $k$ is a constant. Finally, we reduce the problem to the extensible no-$k$-in-line problem, showing that if a positive-density point-set exists for a constant limiter function, then one also exists for any sufficiently regular function $k(n)$.
\end{abstract}

\section{Introduction}

The \textit{no-three-in-line} problem asks about the largest number of points that can be placed in a square lattice of size  $n\times n$ without three of them lying on the same line. It is one of the oldest and most studied problems concerning lattice configurations. The question was introduced by {Dudeney} in 1900 for the case $n=8$ \cite{dudeney-1900-weekly-dispatch} \cite{dudeney1958amusements}. A trivial upper bound for the number of placeable points is $2n$, by the pigeonhole principle; placing $2n+1$ points on an $n\times n$ lattice would result in three points lying on the same vertical line. In small cases up to $n=56$ and for $n=58$, a construction containing $2n$ points is known from the work of Flammenkamp and Prellberg\cite{oeis-A272651}. The best known lower bound of $1.5n(1-o(1))$ was established by Hall, Jackson, Sudbery, and Wild \cite{hall1975some}. However, there is a lack of consensus on the asymptotic number of placeable points, $1.5n,\left(\frac{2\pi^2}{3}\right)^{\frac13}n\approx1.87n$ and $2n$ are all among the suggested values \cite{green100,Brass2005,Guy/Kelly:1968}.

A natural generalization of the previous problem is the \textit{no-($k+1$)-in-line} problem, which permits at most $k$ points per line. The trivial upper bound in this case is $kn$ by similar arguments. This problem was solved very recently by Kovács, Nagy and Szabó \cite{kovacs2025settling} who proved  the existence of a construction of size $kn$ using probabilistic and combinatorial techniques.  Based on this construction, Grebennikov and Kwan \cite{grebennikov2025nok1inlineproblemlargeconstant}  subsequently improved the bound on $k$ and show that  similar statement holds for $k>10^{37}$.

Let $[1,n]^2$ denote the set of lattice points in the square with vertices $(1,1), (1, n), (n, 1), (n, n)$. \\
Erde introduced the \textit{extensible no-three-in-line} problem at the Third Southwestern German Workshop
on Graph Theory. The problem asks whether $$\displaystyle\liminf_{n\to \infty}\frac{\left|S\cap[1,n]^{2}\right|}{n}=0$$
holds for every set $S\subset \mathbb{N}^2$ of points in general position. Nagy, Nagy and Woodroofe \cite{nagy2023extensible} proved that there exists a set $S$ of points in general position satisfying $$\displaystyle\liminf_{n\to \infty}\frac{\left|S\cap[1,n]^{2}\right|}{n/\log^{1+\varepsilon}n}>0.$$
This was improved recently by Ghosal by a $\sqrt{\log n}$ factor to $\Omega(n/\sqrt{\log n})$~\cite{ghosal2026note}.
Nagy, Nagy and Woodroofe also provided computational evidence for values up to $n=10000$, suggesting the existence of a point set $S$, satisfying $$\displaystyle\liminf_{n\to \infty}\frac{\left|S\cap[1,n]^{2}\right|}{n}>0.8.$$ 

\noindent In this article, we examine the \textit{extensible no-$(k(n)+1)$-in-line problem}. This problem resembles the extensible-no-$k$-in-line, however, instead of a constant $k$, we limit the number of points on each line in $[1,n]^2$ by $k(n)$, a monotone increasing function of $n$. To evaluate point sets, we define a density function measuring the size of S, relative to the trivial upper bound.

\begin{definition}[Density]
\label{density}
 Let  $S\subseteq\nn^+\times \nn^+$ be a set of lattice points
 The density of the point set is defined as:
\[
\liminf_{n\to \infty}\frac{\left|S\cap[1,n]^{2}\right|}{k(n)n}.
\]
\end{definition}

\begin{problem}[Main problem]\label{que:Erde}
Find the largest achievable density defined in Definition \ref{density} given that no line admits more than $k(n)$ points in $[1,n]^2$.

\begin{remark}
     We assume that $k(n)$ is monotone without loss of generality, as any non-monotonic $k(n)$ can be replaced by $k'(n)=\inf_{m>n}k(m)$ without changing the set of valid configurations.
    \noindent
\end{remark}

\end{problem}

In Section~\ref{positive_density_sets}, we investigate the case of linear $k(n)$ and prove the following theorem:

\begin{theorem} \label{bigtheorem_linear}
    Let $k(n) = c\cdot n$ for some $0<c<1$. Then there exists a set of points $S\subseteq\nn^+\times \nn^+$ such that
    $$\displaystyle\liminf_{n\to\infty} \frac{|S\cap [1,n]^2|}{k(n)\cdot n} = 1$$ 
    and each line contains at most $k(n)$ points in $S\cap [1, n]^2$.
\end{theorem}

We further verify that if $k(n) = cn^{\alpha}$ with $\alpha \le 1$ and constant $c\in\mathbb{R}$, then a positive lower bound is achievable, meaning the limit inferior is strictly positive.

\begin{theorem} \label{posliminf_linear}
    Suppose that $k(n) = cn^\alpha$ for some $0<\alpha\le1$, $c\in \mathbb{R^+}$. Then, there exists a set $S$ satisfying $$\liminf \frac{|S\cap [1,n]^2|}{k(n)\cdot n} > 0$$
    and no $k(n)+1$ points are collinear within $[1,n]^2$ for all $n\in \nn$.
\end{theorem}

\medskip

In Section~\ref{criteria}, we prove that $\liminf$ values close to $1$ can be achieved only if $k(n)$ increases sufficiently fast.
\begin{theorem}
    Any configuration satisfying $$\displaystyle\liminf_{n\to\infty}\frac{\left|S\cap[1,n]^{2}\right|}{n k(n)} \ge 0.897$$ must satisfy $k(n) = \Omega( n^c)$, where $c>0$ is a constant.
\end{theorem}
 
  As $k(n)\equiv k \notin \Omega( n^c)$ for any constant $c>0$, this result provides an upper bound for Erde's problem: $\displaystyle\liminf_{n\to \infty}\frac{\left|S\cap[1,n]^{2}\right|}{2n}<0.897$.  Computational evidence of Nagy, Nagy and Woodroofe suggests that $\displaystyle\liminf_{n\to \infty}\frac{\left|S\cap[1,n]^{2}\right|}{n}>0.8$, resulting in $\displaystyle\liminf_{n\to \infty}\frac{\left|S\cap[1,n]^{2}\right|}{2n}>0.4$ in our case, leaving a significant gap between the theoretical and experimental results. By generalizing the technique used in the proof, we show that if a construction has density 1, then $k(n)=\Omega(n^c)$ for any constant $c<1$.

In Section 4, we establish a direct connection between the classical extensible no-three-in-line problem and our generalized no-$(k(n)+1)$-in-line variant. We introduce the concept of \textit{sublinear limiter functions} to formalize the notion of a sufficiently regular function, and prove that a positive-density solution to Erde's original question yields a universal lower bound for these limiters. Specifically, our results show that a positive answer to Erde's question for $k=2$ would imply the existence of a positive-density set for any sublinear limiter $k(n)$. This reduces the generalized problem to the classical case, reinforcing the suggestion that the difficulty of this problem increases as the growth rate of $k(n)$ decreases.

\subsection*{Acknowledgements}

The results presented in this paper are the outcome of the Research Experience for Undergraduates (REU) program at ELTE in 2025 as a joint effort by the authors. We express our gratitude to the organizers and the manager of the program and especially to our supervisor Zoltán Lóránt Nagy.

\newpage
\section{Sets with positive density}\label{positive_density_sets}

In this section we prove Theorem \ref{bigtheorem_linear} by providing a point-set with density of 1 to the $k(n) = c\cdot n$ problem. We further show a point-set with positive density to the $k(n) = c\cdot n^\alpha$ problem. 
\medskip

\subsection{Proof of Theorem \ref{bigtheorem_linear}}

The proof of Theorem \ref{bigtheorem_linear} consists of multiple steps, gradually extending the proof to all values of $0<c<1$. We prove this theorem first for $c=\frac{1}{m}, m\in \mathbb{N}^+$, second for rational $c$ and finally for all $0<c<1$.
\begin{claim}\label{claim1perm}
   There exists a set of points $S\subseteq\nn^+\times \nn^+$ such that
    $$\displaystyle\liminf_{n\to\infty} \frac{|S\cap [1,n]^2|}{\frac{n}{m}\cdot n} = 1$$
    for all $m\in\nn$ and each line contains at most $\frac{n}{m}$ points in $S\cap [1,n]^2$.
\end{claim}
\begin{proof}
We provide a constructive proof for the statement, by defining an infinite set of lattice points, $S$ that satisfies \cref{claim1perm}. Consider the following set of points $S$: \\
In every row, $m$ consecutive points are selected then the subsequent $(m-1)\cdot m$ points are skipped. 
This selection process is repeated along the entire row. The $(i+1)$th row is constructed by shifting the first row by $m\cdot i$ to the right. A more rigorous definition is provided below.
\begin{definition}
Let $A=2m^2$ be a constant.
$$S:=\{(i, j)\in \nn_{>A}\times \nn_{>A} | \    i-mj \in \{0, 1, ..., m-1\} \mod m^2\}.$$
\end{definition}
\begin{note}
    No point of $S$ lies in any row or column with index in $[1,A]$.
\end{note}
\begin{note}
    As the defining formula suggests, $S$ is periodic with period $m^2$ along both the $x$- and $y$-axes and has $m$ points in each period.
\end{note} \noindent

    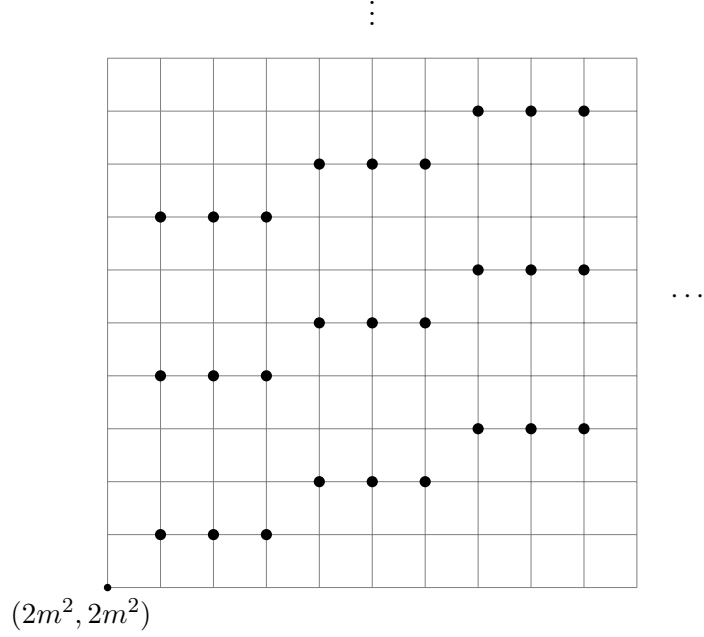
\begin{figure}
    \centering
    \begin{tikzpicture}[scale=0.7]
		\draw[step=1cm,gray,very thin] (0,0) grid (10,10);
		
		\foreach \x/\y in {1/1, 2/1, 3/1, 4/2, 5/2, 6/2, 7/3, 8/3, 9/3, 1/4, 2/4, 3/4, 4/5, 5/5, 6/5, 7/6, 8/6, 9/6, 1/7, 2/7, 3/7, 4/8, 5/8, 6/8, 7/9, 8/9, 9/9} {
			\fill[black] (\x,\y) circle[radius=3pt];
		}
		
		\fill[black] (0, 0) circle[radius=2pt];
		\node at (-0.5, -0.5) {$(2m^2, 2m^2)$};
          
        \node at (5, 11) {$\vdots$};
        \node at (11, 5.5) {$\cdots$};
	\end{tikzpicture}
    \caption{for $c = \frac13$}
    \end{figure}
    
\noindent
In the following, we will prove that $$\displaystyle\liminf_{n\to\infty} \frac{|S\cap [1,n]^2|}{\frac{n}{m}\cdot n} = 1$$
and $l\cap S$ has a maximum of $\frac{n}{m}$ points in $[1,n]^2$ for every line $l$.

   For a fixed $n$, each row (excluding the first $A$ rows) contains at least $\frac n m-m-A$ points in $[1,n]^2$. Consequently, within square $[1, n]^2$, the set $S$ contains at least $n\left(\frac{n}{m} - A-m\right)-An$ points. Therefore
   $$\liminf_{n\to \infty} \frac{|S\cap [1,n]^2|}{\frac{n}{m}\cdot n} \ge \liminf_{n\to \infty} \frac{\frac{n^2}{m} - 2An-nm}{\frac{n^2}{m}}=\liminf_{n\to \infty}\left(1-\frac{(2A+m)m}{n}\right)=1$$
   as $A$ and $m$ are fixed.

We prove that any line $l$ contains at most $\frac{n}{m}$ points of $S\cap [1,n]^2$ by considering the possible slopes $q$ of $l$.
\\
\noindent
\bf Case 0: \it $l$ is horizontal or vertical. \rm \\
If $l$ is horizontal, it contains at most $$\lceil(n-A) /m^2\rceil m  + m=\lceil(n-2m^2) /m^2\rceil m  + m= $$
    $$=\lceil n /m^2 - 2\rceil m  +m\le \left( n/m^2-1\right) m + m=n/m$$
    points, since every horizontal line contains $m$ selected points from every consecutive block of $m^2$ points.\\
    If $l$ is vertical, it contains every $m$th lattice point, except within the first $A$ rows. Hence, the number of points in $l\cap [1,n]^2$ is at most
    $$\lceil (n-A)/m\rceil\le \frac{n}{m}.$$
\bf Case 1: $-1\le q\le 1, q\not =0$. \\
\rm Consider a general line  $l(x) = qx +b$ with slope $q$ satisfying $-1\le q\le 1$ and $q\not=0$. 
Partition the lattice $\nn^+\times\nn^+$ into vertical stripes $v_0,v_1, ...$, each of width $m$. The $k$th stripe, denoted by $v_k$ consists of all points $$[km, (k+1)m-1]\times\nn.$$ 
Observe that within each stripe $v_k$, the set $S\cap v_k$ consists of horizontal segments of length at most $m$, separated vertically by $m-1$ empty rows.\\
Further observe that if the slope of $l$ satisfies $|q|\le 1$, then $l$ can intersect at most one of the rows in each stripe $v_k$ that contains points of $S$, since these rows have a distance of $m$ units vertically and the stripe has width $m-1$ unit horizontally. Hence, in every stripe $v_k$, the line $l$ intersects $S$ in at most one point. The square $[1,n]^2$ intersects at most $$\left \lceil\frac{n}{m}\right \rceil<\frac{n}{m}+1$$ vertical stripes. Moreover, $v_0, ..., v_{m-1}\in[1, A]\times \nn$ as each of them has a width of $m$. Therefore, $S\cap [1,n]^2$ does not intersect them. Hence, the line $l$ contains at most $$\frac{n}{m}+1-m\le\frac{n}{m}$$ points of $S$ inside $[1,n]^2$.
\\\\
\noindent
\bf Case 2: \rm $1< q < +\infty$. Next, consider lines $l(x)=qx+b$ with $1< q < +\infty$. Let the points of $S\cap l\cap [1,n]^2$ be $l_1, ..., l_d$ sorted according to their $x$ coordinates. We aim to show that $d \le \frac{n}{m}$. \\
Denote the $x$ and $y$ coordinates of any point $P$ by $x(P), y(P)$ respectively.
To prove $d\le \frac{n}{m}$, it is sufficient to show that $y(l_i)+m\le y(l_{i+1})$ for all $i<d$. Clearly, if $l_i$ and $l_{i+1}$ lie within the same vertical stripe $v_j$, then $y(l_{i+1})-y(l_i)\ge m$, since the points cannot belong to the same row. If $l_i\in v_j, l_{i+1}\in v_{j+1}$, then $y(l_{i+1})\ge y(l_i)+(m+1)$ because it is impossible to have $y(l_{i+1})=y(l_i)+1$ when $q>1$. \\
If there is at least one empty vertical stripe $v_j$ between $l_i$ and $l_{i+1}$ then the $y$-coordinate of $l_{i+1}$ is at least $v_i+m$ because it increases by $m$ in $v_j$ alone. This completes the proof for lines with slope $q>1$.
\\\\
\noindent
\bf Case 3: \rm $-\infty < q < -1$.\\
We now apply an argument similar to Case 2. Consider the slope $q$ of $l$.
If the vertical distance between any two consecutive lattice points on $l$ is at least $m$, that is $y(l_i)-y(l_{i+1})\ge m$, then at most $\left \lfloor \frac nm \right \rfloor$ points can fit vertically in $[1,n]^2$. A vertical distance between two consecutive points can be smaller than $m$ only if they lie in consecutive stripes, because $y(l_{i+1})=q(x(l_{i+1})-x(l_i))+y(l_i)\le-1\cdot m+y(l_i)$ if $l_i$ and $l_{i+1}$ were separated by an empty vertical stripe. If they lie in consecutive stripes and $y(l_i)-y(l_{i+1})<m$ then the vertical distance is $m-1$ (see figure). Therefore, $$q=\frac{y(l_{i+1})-y(l_i)}{x(l_{i+1})-x(l_i)}=-\frac{m-1}{r},$$ where $r\in \mathbb{Z}^+$.

On the other hand, if $[1,n]^2$ contains more than $\frac{n}{m}$ points, then there must be a stripe containing multiple points, as each stripe has a width of $m$. If two consecutive points, $l_i$ and $l_{i+1}$ lie in the same stripe, then $$q=\frac{y(l_{i+1})-y(l_i)}{x(l_{i+1})-x(l_i)}=\frac{-am}{b}$$ for some $a\in \mathbb{Z}^+$ and integer $0<b<m$. Combining these, we get $\frac{-am}{b}=q=-\frac{m-1}{r}$ so $b(m-1)=amr$. Since $\gcd(m, m-1)=1$, it follows that $m|b(m-1)$ which implies $ m|b$. However this contradicts $0<b<m$. 
We can conclude that if $q<-1$ then $l$ contains at most $\frac{n}{m}$ points in $[1,n]^2$
\end{proof}
\bigskip
\begin{claim} \label{claimpperq}
    There exists a set of points $S\subseteq\nn^+\times \nn^+$ such that
    $$\displaystyle\liminf_{n\to\infty} \frac{|S\cap [1,n]^2|}{\frac{pn}{q}\cdot n} = 1$$
    for all $p,q \in \nn, \frac{p}{q}<1$ and each line contains at most $\frac{n}{m}$ points in $S\cap [1,n]^2$.
    $c=\frac{p}{q}<1$, for $p,q \in \nn$, $(p,q) = 1$.
\end{claim}
\begin{proof}
    We now extend the configuration constructed for the case $c=\frac{1}{m}$ to all rational values $c\in[0,1]$.\\
    For a rational number $c=\frac{p}{q}$, the construction is defined as the union of $p$ disjoint copies of the construction for $c=\frac{1}{q}$.\\
    Let this configuration be denoted by $S$. As before, the first $A$ rows and $A$ columns of $S$ are empty (where $A$ is the same constant used in the $c = 1/q$ case). In $[A+1, \infty)^2$ we define $$S := \{(i, j)\  |\  i-qj \in \{ 0, ..., pq-1\} \mod q^2\}.$$ 
    With this construction, in every square $[1,n]^2$, the number of points lying on a given line is $p$ times larger than in the configuration for $c=\frac{1}{q}$, because $S$ is created as the union of $p$ disjoint copies of the $c=\frac{1}{q}$ construction. Consequently, the density of $S$ is also $p$ times higher than in the $c=\frac{1}{q}$ case.
\end{proof}

\begin{claim} \label{claimirrac}
    There exists a set of points $S\subseteq\nn^+\times \nn^+$ such that
    \begin{equation} \label{eqirrac}
        \displaystyle\liminf_{n\to\infty} \frac{|S\cap [1,n]^2|}{c\cdot n} = 1
    \end{equation}
    for all irrational $0<c<1$ and each line contains at most $cn$ points in $S\cap [1,n]^2$.
\end{claim} 

\begin{proof}
We provide a configuration $S$ that satisfies \eqref{eqirrac}. The configuration will be defined recursively, based on the digits of $c$ in base 10. \\ Denote the number formed by the first $k$ digits of $c$ by $c_k$.
 Let $(s_k)_{k\ge1}$ be a strictly increasing sequence of positive integers, which will be described later. \\
In each step, we extend the previously constructed configuration corresponding to $c_k$ within square $[1, s_k]^2$ to obtain a configuration for $c_{k+1}$ in the larger square $[1,s_{k+1}]^2$. \\
The $k$th extension consists of two components: a gap, denoted by $G_k$, and a part containing new points, denoted by $S_k$. Together with the previous construction within $[1,s_{k-1}]^2$, these parts completely fill the square $[1,s_k]^2$. Since $s_k\to \infty$, it follows that $$\bigcup_{k\in\nn} [1,s_k]^2$$ covers the entire first quadrant. The width of $G_k$ depends only on $c_k$ and $s_k$ depends on $g_{k+1}$ and $c_k$\\

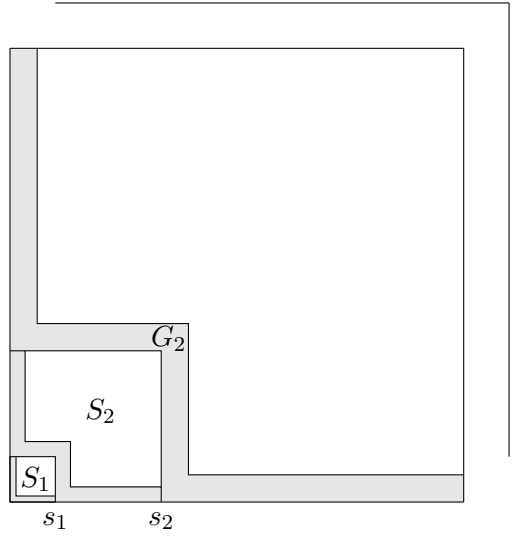
\begin{figure}[!hbt]
\centering
\begin{tikzpicture}[scale = 0.2]
  \node (O) at (0,0) {};

  \def\s{3}      
  \def\cd {1.8}
  \def\cb {1}      
  \def\ca {0.4}   
  \def\rx{7}   
  \def\ry{15}
  \def\w{30}

  \fill[gray!20]
    (0,0) -- (0,\s) -- (\ca,\s) -- (\ca,\ca) -- (\s,\ca) -- (\s,0) -- cycle;

  \fill[gray!20]
    (\s,0) -- (\s+\cb,0) -- (\s+\cb,\s+\cb) -- (0,\s+\cb) -- (0,\s) -- (\s,\s) -- cycle;

  \fill[gray!20]
    (0,\s) -- (0,\s+\rx) -- (\cb,\s+\rx) -- (\cb,\s) -- cycle;

  \fill[gray!20]
    (\s,0) -- (\s+\rx,0) -- (\s+\rx,\cb) -- (\s,\cb) -- cycle;

  \fill[gray!20]
    (\rx + \s, 0) -- (\w , 0) -- (\w, \cd) -- (\rx + \s + \cd, \cd) -- (\rx + \s + \cd, \rx + \s + \cd) -- (\cd, \rx + \s + \cd) -- (\cd, \w) -- (0, \w) -- (0, \rx + \s) -- (\rx + \s, \rx + \s) -- cycle; 

    \draw (0, 30) -- (30, 30);
    \draw (3, 33) -- (33, 33);
    \draw (30, 0) -- (30, 30);
    \draw (33, 3) -- (33, 33);

  \draw (0,0) -- (\w,0);
  \draw (0,0) -- (0,\w);

  \draw (0,0) rectangle (\s,\s);
  \draw (\s, \ca) -- (\ca, \ca) -- (\ca, \s);
  \draw (0, \s + \rx) -- (\cb, \s + \rx) -- (\cb, \s + \cb) -- (\s + \cb, \s + \cb) -- (\s + \cb, \cb) -- (\s + \rx, \cb) -- (\s + \rx, 0);
  \draw (\cb, \s + \rx) -- (\s + \rx, \s + \rx) -- (\s + \rx, \cb);
  \draw (\w, \cd) -- (\rx + \s + \cd, \cd) -- (\rx + \s + \cd, \rx + \s + \cd) -- (\cd, \rx + \s + \cd) -- (\cd, \w);
  \node[below] at (\s, 0) {$s_1$};
  \node[below] at (10, 0) {$s_2$};
  \node at (6, 6) {$S_2$};
  \node at (1.7, 1.5) {$S_1$};
  \node at (10.5, 10.8) {$G_2$};
\end{tikzpicture}
\caption{The gaps in grey and $S_k$ in white}
\label{fig: emptygaps}
\end{figure}

Now describe the configuration $S$ inductively.\\
Since each $c_k$ is a positive rational number, it can be expressed as the ratio of two positive integers. As $c_k$ represents the truncation of $c$ after the $k$th digit, it can be written in the form of a fraction of two positive integers, whose denominator is $10^k$. Denote this denominator by $m_k$, that is $m_k=10^k$.~\\\\
To begin the construction, we create a gap $G_1$, consisting of the first $g_1=2m_1^2$ rows and columns. The points within any gap are permanently excluded from the selection throughout the entire construction of $S$.~\\\\
Next, within square $[g_1+1,s_1]^2$ let our configuration, be $S\cap S_1$ where $S_1$ is the configuration  for $c_1$,
as constructed in \cref{claimpperq} for the rational numbers of $[0,1]$.

Now proceed with the induction. 
Assume that we have already constructed $S$ within $[1,s_k]^2$ such that each line contains at most $c\cdot s_k$ points inside $[1,s_k]^2$,  $$\frac{|S\cap [1,s_k]^2|}{(s_k+g_{k+1})^2}\ge c_{k-1},$$ $m_{k}^2\mid s_k$ and $m_{k+1}^2\mid s_k +g_{k+1}$.

As $$\liminf \frac{|S_k\cap [1,n]^2|}{n^2}=c_k>c_{k-1}, $$ we can choose $s_k$ large enough to satisfy the previous inequality. By choosing $g_{k+1}\ge 2m_k^2+2m_{k+1}^2$, we can limit the density of points on a single line. \\

\begin{lemma} \label{lemma2}
There exists a sequence $(s_k)_{k\ge 1}$ and $(g_k)_{k\ge 1}$ of strictly increasing positive integers such that the density $$\frac{|S\cap [1,n]^2|}{|[1,n]^2|}$$ is at least $c_k$ for all $n\ge s_{k+2}$.
\end{lemma}
\begin{proof}
First, we prove that if
$$\frac{|S\cap [1,s_{k}+g_{k+1}]^2|}{|[1, s_{k}+g_{k+1}]^2|}\ge c_{k-1}$$
then the density of $S$ remains at least $c_{k-1}$ for all $n>s_{k}+g_{k+1}$. We set $s_k$ sufficiently large to ensure that the construction $S\cap S_{k}$ is at least $c_{k-1}$-dense within $[1,n]^2$ for every $s_{k}<n\le s_{k}+g_{k+1}$.

Since for all larger squares (i.e. for $n\ge s_{k+2}$) we provide an even stronger lower bound on the density, namely $c_{k}$, it is sufficient to verify the density condition for the squares $[1,n]^2$ with $s_{k}+ g_{k+1}< n\le  s_{k+1}+g_{k+2}$.

For $s_{k}+ g_{k+1}< n\le  s_{k+1}$ we prove by induction that the density is at least $c_{k-1}$.
By assumption, the density was at least $c_{k-1}$ for $n=s_k+g_{k+1}$. Suppose the density holds for $n$. For $n+1$ consider the rectangles $[1, n+1]\times[n+1,n+1]$ and $[n+1,n+1] \times [1, n+1]$. These are the beginning of the $(n+1)$th row and column so after the first $2m_{k+1}^2$ cells they are periodic by $m_{k+1}^2$ and have a relative density of $c_{k+1}$. Therefore the density of points in $[1, n+1]\times[n+1,n+1]\cup [n+1,n+1] \times [1, n+1]$ converges to $c_{k+1}$ meaning that for $n$ large enough the density is at least  $c_k>c_{k-1}$. This means that the entire $S$ has a density of at least $c_{k-1}$ in $[1,n+1]^2$ completing the induction. Setting $s_k$ large enough ensures that $n$ is sufficiently large for this density.
\end{proof}
\noindent
For the proof of lemma \ref{lemma3} we choose $s_{k+1}$ such that $s_{k+1}>s_k^2$. This choice clearly does not violate or contradict anything used about $s_{k+1}$ during the proof of lemma \ref{lemma2}.
\begin{note}
When selecting $s_k$ for the construction of $S\cap S_{k+2}$, we must take into account the constraints imposed on $s_k$ during the construction of $S_{k+1}$; we set $s_k$ as the maximum of these two requirements and $s_{k-1}^2+1$.
\end{note}

\begin{figure}[ht]
    \centering
    \begin{subfigure}[b]{0.48\textwidth}
        \centering
        \begin{tikzpicture}[scale = 0.15] 
            \node (O) at (0,0) {};
            \def\s{3} \def\cd {1.8} \def\cb {1} \def\ca {0.4} \def\rx{7} \def\w{30}
            \def\aab{1} \def\aac{7.27} \def\aad{2.48}
            
            \fill[gray!20] (0,0) -- (0,\s) -- (\ca,\s) -- (\ca,\ca) -- (\s,\ca) -- (\s,0) -- cycle;
            \fill[gray!20] (\s,0) -- (\s+\cb,0) -- (\s+\cb,\s+\cb) -- (0,\s+\cb) -- (0,\s) -- (\s,\s) -- cycle;
            \fill[gray!20] (0,\s) -- (0,\s+\rx) -- (\cb,\s+\rx) -- (\cb,\s) -- cycle;
            \fill[gray!20] (\s,0) -- (\s+\rx,0) -- (\s+\rx,\cb) -- (\s,\cb) -- cycle;
            \fill[gray!20] (\rx + \s, 0) -- (\w , 0) -- (\w, \cd) -- (\rx + \s + \cd, \cd) -- (\rx + \s + \cd, \rx + \s + \cd) -- (\cd, \rx + \s + \cd) -- (\cd, \w) -- (0, \w) -- (0, \rx + \s) -- (\rx + \s, \rx + \s) -- cycle; 
            
            \draw (0,0) -- (\w,0); \draw (0,0) -- (0,\w);
            \draw (0,0) rectangle (\s,\s);
            \draw (\s, \ca) -- (\ca, \ca) -- (\ca, \s);
            \draw (0, \s + \rx) -- (\cb, \s + \rx) -- (\cb, \s + \cb) -- (\s + \cb, \s + \cb) -- (\s + \cb, \cb) -- (\s + \rx, \cb) -- (\s + \rx, 0);
            \draw (\cb, \s + \rx) -- (\s + \rx, \s + \rx) -- (\s + \rx, \cb);
            \draw (\w, \cd) -- (\rx + \s + \cd, \cd) -- (\rx + \s + \cd, \rx + \s + \cd) -- (\cd, \rx + \s + \cd) -- (\cd, \w);
            
            \draw [color = red] (\s-0.5, 0) -- (\cd + \s + 12, \w);
            \draw [color = blue, fill={rgb,255:red,173; green,216; blue,230}, fill opacity=0.3] (\aad, 0) -- (\aad + \aab, 0) -- (\aad + \aab, \aab) -- (\aad, \aab) -- (\aad, 0);
            \draw [color = blue, fill={rgb,255:red,173; green,216; blue,230}, fill opacity=0.3] (\s + 0.05, \aab + 0.05) -- (\s + 0.05, \s + \rx - 0.05) -- (\s + \s + \rx - \aab, \s + \rx - 0.05) -- (\s + \s + \rx - \aab, \aab + 0.05) -- (\s + 0.05, \aab + 0.05);
            \draw [color = blue] (\aac, \w) -- (\aac, \s + \rx) -- (\w, \s + \rx);
            \fill [color ={rgb,255:red,173; green,216; blue,230}, opacity=0.3] (\aac, \s + \rx) -- (\aac, \w) -- (\w, \w) -- (\w, \s + \rx) -- (\aac, \s + \rx);
            \draw [dashed] (\s + \rx + \s + \rx, 0) -- (\s + \rx + \s + \rx, \s + \rx + \s + \rx) -- (0, \s + \rx + \s + \rx);
            \fill [color = blue, fill={rgb,255:red,173; green,216; blue,230}, fill opacity=0.6] (\aac, \s + \rx) -- (\aac, \s + \rx + \s + \rx) -- (\s + \rx + \aac, \s + \rx + \s + \rx) -- (\s + \rx + \aac, \s + \rx) -- (\aac, \s + \rx);
        \end{tikzpicture}
        \caption{Squares containing line $l$}
        \label{fig:nzetek2}
    \end{subfigure}
    \hfill
    \begin{subfigure}[b]{0.48\textwidth}
        \centering
        \begin{tikzpicture}[scale = 0.15]
            \node (O) at (0,0) {};
            \def\s{3} \def\cd {1.8} \def\cb {1} \def\ca {0.4} \def\rx{7} \def\w{30}
            \def\aab{3.82} \def\aac{6} \def\aad{18.8} \def \dd{0.27} \def \cc{1.45}
            
            \fill[gray!20] (0,0) -- (0,\s) -- (\ca,\s) -- (\ca,\ca) -- (\s,\ca) -- (\s,0) -- cycle;
            \fill[gray!20] (\s,0) -- (\s+\cb,0) -- (\s+\cb,\s+\cb) -- (0,\s+\cb) -- (0,\s) -- (\s,\s) -- cycle;
            \fill[gray!20] (0,\s) -- (0,\s+\rx) -- (\cb,\s+\rx) -- (\cb,\s) -- cycle;
            \fill[gray!20] (\s,0) -- (\s+\rx,0) -- (\s+\rx,\cb) -- (\s,\cb) -- cycle;
            \fill[gray!20] (\rx + \s, 0) -- (\w , 0) -- (\w, \cd) -- (\rx + \s + \cd, \cd) -- (\rx + \s + \cd, \rx + \s + \cd) -- (\cd, \rx + \s + \cd) -- (\cd, \w) -- (0, \w) -- (0, \rx + \s) -- (\rx + \s, \rx + \s) -- cycle; 
            
            \draw (0,0) -- (\w,0); \draw (0,0) -- (0,\w);
            \draw (0,0) rectangle (\s,\s);
            \draw (\s, \ca) -- (\ca, \ca) -- (\ca, \s);
            \draw (0, \s + \rx) -- (\cb, \s + \rx) -- (\cb, \s + \cb) -- (\s + \cb, \s + \cb) -- (\s + \cb, \cb) -- (\s + \rx, \cb) -- (\s + \rx, 0);
            \draw (\cb, \s + \rx) -- (\s + \rx, \s + \rx) -- (\s + \rx, \cb);
            \draw (\w, \cd) -- (\rx + \s + \cd, \cd) -- (\rx + \s + \cd, \rx + \s + \cd) -- (\cd, \rx + \s + \cd) -- (\cd, \w);
            
            \draw [color = red] (\s+ \rx + \cd + 3, 0) -- (0, \s + \rx + \cd + 7);
            \draw [color = blue, fill={rgb,255:red,173; green,216; blue,230}, fill opacity=0.3](\s + \rx + \ca , \aac - \ca) -- (\s + \rx + \ca , \s + \rx) --(\aac , \s + \rx) -- (\aac, \aac - \ca) -- (\s + \rx + \ca , \aac - \ca);
            \draw [dashed] (0, \aad - 4.5) -- (\aad - 4.5 , \aad - 4.5) -- (\aad - 4.5, 0);
            \draw [color = blue, fill={rgb,255:red,173; green,216; blue,230}, fill opacity=0.3] (6.53 -\dd, 10.8) -- (0, 10.8) -- (0, 17.33 - \dd) -- (6.53 - \dd, 17.33 - \dd) -- (6.53-\dd, 10.8);
            \fill [color = {rgb,255:red,173; green,216; blue,230}, opacity = 0.6] (6.53 -\dd, 10.8) -- (6.53 -\dd, \aad - 4.5) -- (6.53 - \dd -\aad + 4.5 + 10.8, \aad - 4.5) -- (6.53 - \dd -\aad + 4.5 + 10.8, 10.8);
            \draw [color = blue, fill={rgb,255:red,173; green,216; blue,230}, fill opacity=0.3] (10.8, 6.53 -\cc) -- (10.8, 0) -- (17.33 - \cc, 0) -- (17.33 - \cc, 6.53 - \cc) -- (10.8, 6.53-\cc);
            \fill [color = {rgb,255:red,173; green,216; blue,230}, opacity = 0.6] (10.8, 6.53 -\cc) -- (\aad - 4.5, 6.53 -\cc) -- (\aad - 4.5, 6.53 - \cc -\aad + 4.5 + 10.8) -- (10.8, 6.53 - \cc -\aad + 4.5 + 10.8);
        \end{tikzpicture}
        \caption{Line $l$ with negative gradient}
        \label{fig:negline}
    \end{subfigure}
    
    \caption{Comparison of positive and negative gradients}
    \label{fig:comparison}
\end{figure}
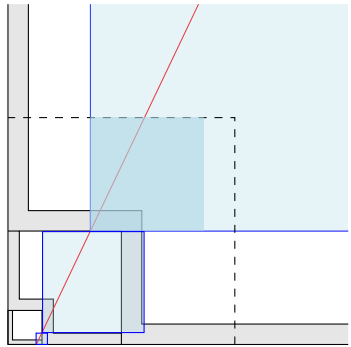
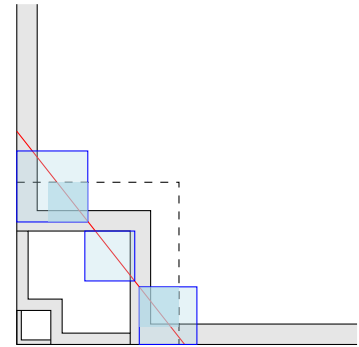

\begin{lemma} \label{lemma3}
Each line $l$ contains at most $c_k\cdot n$ points in $S\cap [1,n]^2$ for every $n\le  s_k$ (where $n,k\in\nn$).

\end{lemma}
\begin{proof}
We prove the lemma by considering the the slope $q$ of $l$.

\noindent
\bf Case 1: \rm $|q|\ge 1$ or $l$ is vertical. In this case we cover $l$ by disjoint squares that are vertically above each other. Each square is congruent (up to rotation by 180 degrees and reflection) to the part of a rational construction $S_k$ in $[1,n]^2$. While these squares may differ from the part they cover in the actual construction $S$, they contain the same points on $l$ as $S$ does.

Traversing from bottom to top, $l$ first enters $S_k$ through a gap of width at least $2m_k^2+2m_{k-1}^2$. We select the covering square $\Sigma$ such that one of its vertices is the bottom entry point of $l$ into the $2m_k^2$-wide neighbourhood of $G_k$ (or $G_{k+1}$, depending on which gap $l$ crosses upon entering $S_k$ from the bottom). The side length of $\Sigma$ is chosen so that it ends vertically $2m_k^2$ units above the exit point of $l$ from $S_k$. Each gap is at least as wide as the sum of the gap widths corresponding to the constructions it separates. Any two consecutive covering squares are strictly ordered vertically, covering different parts of the separating gap, with the top of the $k$th square lying below the bottom of the $(k+1)$th square.

If $l$ enters $S_k$ multiple times (see Figure \ref{fig:negline}), we assign a distinct square to each section of $l$ passing through $S_k$. Crucially, exactly one square covers each section of $l$.

Because each square is congruent to $S_k\cap[1,h]^2$, it contains at most $c_k \cdot h$ points on any line. Upon entering $\Sigma$, $l$ first passes through a gap of width $2m_k^2$. Relative to $l$, $\Sigma$ is geometrically equivalent to $S_k\cap [1,h]^2$, meaning that $l$ contains at most $c_k \cdot h$ points within $\Sigma$. Because squares do not overlap vertically, the sum of their side-length is at most $n$. Therefore the total number of points of $l$ in $[1,n]^2$ is $\sum_j c_{k_j}\cdot h_j\le \sum_j c\cdot h_j\le c\cdot n$.

\noindent
\bf Case 2: \rm $|q|<1$. In this case we cover $l$ by horizontally disjoint squares exactly as in Case 1.
\end{proof}
\end{proof}
\noindent Claim \ref{claim1perm}, Claim \ref{claimpperq} and Claim \ref{claimirrac} together fully prove Theorem \ref{bigtheorem_linear}.

\medskip
We can use the proof technique from Claim \ref{claimirrac} to verify the following lemma. 
We will use Lemma \ref{lowerboundlemma} in Section~\ref{criteria} to construct functions $k(n)$ that give $\liminf \frac{|S \cap [1,n]^2|}{n k(n)}=c$ for any given $c\in(0,1)$.\\
\begin{lemma}\label{lowerboundlemma}
For any given $k(n) \le n$, there exists a set $S$ containing no $k(n) + 1$ collinear points, which satisfies
$$\liminf_{n\to\infty} \frac{|S \cap [1,n]^2|}{k(n)\cdot n} \geq \liminf_{n\to\infty} \frac{k(n)}{n}.$$ 
\end{lemma}
\begin{proof}
We use a slightly modified version of the construction presented in the previous irrational case. ~\\
\noindent
Let $c:= \liminf \frac{k(n)}{n}>0$ and $(c_k)_{k=1}^{\infty}$ be a monotone increasing sequence of rationals with $\lim_{k\to\infty}c_k=c$. It is possible to choose a suitable sequence $(c_k)$ even if $c$ is rational.
After setting $s_k$, $s_{k+1}$ can be chosen arbitrarily large, imposing additional constraints on  its minimal size while preserving the validity of the proof used in Claim \ref{claimirrac} for irrationals. Whenever $n>s_{k-1}$ for the new, increased sequence of $(s_k)$, 
$\frac{k(n)}{n}>c_{k}$ holds. Choose the sequence $(s_k)$ to be monotone increasing, such that $\lim_{k\to \infty} s_k=\infty$. 

\noindent
By Lemma~\ref{lemma3}, each line $l$ contains at most $c_k\cdot n<\frac{k(n)}{n}n=k(n)$ points in $S\cap [1,n]^2$ for every $s_{k-1}<n\le  s_k$. This proves that our construction is correct and provides a set of points $S$ such that
$$\liminf_{n\to\infty} \frac{|S\cap [1,n]^2|}{c n^2} = 1\text{, which implies}$$
$$\liminf_{n\to\infty} \frac{|S\cap [1,n]^2|}{ k(n)\cdot n}\ge\liminf_{n\to\infty} \frac{|S\cap [1,n]^2|}{ n^2} = c=\liminf_{n\to\infty}\frac{k(n)}{n}.$$

We used $k(n) \le n$ for every $n$ in the inequality of the last line.

\end{proof}

As a generalization, we provide a set with positive $\liminf$ for any polynomially growing functions $k(n)$. Determining the optimal achievable density, however, remains an open problem.

\subsection{Proof of Theorem \ref{posliminf_linear}}

In the following, we prove Theorem \ref{posliminf_linear}.
    We first construct a suitable point-set $S$, then prove that $S$ is dense, and finally verify that it satisfies the no-$(k(n)+1)$-in-line property.

    \begin{proof}[Proof of Theorem \ref{posliminf_linear}]
    We begin with the construction of the point set $S$ defined above.
    As in the previous constructions, $S$ only contains points with sufficiently large $x$ coordinates, specifically, we set the first $N$ columns of $S$ be empty.
    
    Let $A>0$ be a scaling parameter whose value will be chosen to be $4$ later. 
     We place the selected points into towers, and let the index of these towers be denoted by $m$. For each $m>M \in \Z^+$ we build a tower (red rectangles in Figure \ref{nalpha}) to $[(\frac{Am}{c})^{1/\alpha}, (\frac{Am+1}{c})^{1/\alpha}]$ on the x-axis of width 
    $$w_m = c^{-\frac1\alpha}\left((Am+1)^{1/\alpha} - ({Am})^{1/\alpha}\right)$$
    and height of 
    $$
    h_m = \frac{1}{4}c^{-\frac1\alpha}(Am)^{1/\alpha}
    $$
    where $0<\alpha <1$.
    Of course, these expressions only make sense when $m$ is sufficiently large.

    Divide the towers into squares. Each tower is filled with squares of side-length $w_m$ on top of each other. We fill each tower with the maximal number of squares, denote this by 
    $$
    s_m =\left\lfloor\frac{h_m}{w_m}\right\rfloor
    $$
    In each square we put a modular hyperbola,  which guarantees at least $\frac32w_m - o(1) > w_m$ points in general position in each square if $m$ is large enough \cite{hall1975some}.

\begin{figure}[h!!]
\centering
\begin{tikzpicture}[scale = 0.1]
        \node at (47, 6) {$\textcolor{red}{h_5}$};
        \node at (60, -4) {$\textcolor{red}{w_6}$};
	  \draw[->, thick] (0,0) -- (90,0);
	   \draw[->, thick] (0,0) -- (0,45);
	
	  \draw[thick, domain=0:90] plot(\x, {0.25*\x}) node[right] {$y = 0.25x$};

	    \filldraw[blue!30, draw=black] (25.6,0) rectangle (28.9,3.3);
        \draw[draw=red] (25.6,0) rectangle (28.9,6.4);
	    \node at (25, -4) {$16^2$};
	    \node at (30, -8) {$17^2$};
	    \draw[thick] (25.6,-1) -- (25.6,1);
	    \draw[thick] (28.9,-1) -- (28.9,1);
	    \fill (27.86, 2.02) circle (5pt);
	    \fill (26.06, 2.59) circle (5pt);
	    \fill (25.69, 0.25) circle (5pt);
	    \fill (26.83, 0.91) circle (5pt);
	    \fill (28.12, 0.15) circle (5pt);
	    \fill (28.71, 3.14) circle (5pt);
	    \fill (28.82, 0.29) circle (5pt);
	    \fill (27.05, 0.11) circle (5pt);
	    \fill (27.82, 3.21) circle (5pt);
	    \fill (28.48, 2.91) circle (5pt);

	    \filldraw[blue!30, draw=black] (40,0) rectangle (44.1,4.1);
	    \filldraw[green!30, draw=black] (40,4.1) rectangle (44.1,8.2);
        \draw[draw = red] (40, 0) rectangle (44.1, 10);
	    \fill (42.20, 2.37) circle (5pt);
	    \fill (42.23, 1.74) circle (5pt);
	    \fill (42.24, 3.41) circle (5pt);
	    \fill (41.42, 1.24) circle (5pt);
	    \fill (43.37, 1.98) circle (5pt);
	    \fill (43.43, 1.86) circle (5pt);
	    \fill (40.93, 4.05) circle (5pt);
	    \fill (43.15, 2.43) circle (5pt);
	    \fill (43.81, 1.55) circle (5pt);
	    \fill (40.61, 2.68) circle (5pt);
	    \fill (42.44, 6.77) circle (5pt);
	    \fill (42.17, 5.88) circle (5pt);
	    \fill (40.62, 8.19) circle (5pt);
	    \fill (43.87, 4.54) circle (5pt);
	    \fill (43.23, 4.58) circle (5pt);
	    \fill (40.79, 6.09) circle (5pt);
	    \fill (42.45, 5.06) circle (5pt);
	    \fill (42.29, 6.49) circle (5pt);
	    \fill (41.40, 5.77) circle (5pt);
	    \fill (42.95, 4.10) circle (5pt);

	    \filldraw[blue!30, draw=black] (57.6,0) rectangle (62.5,4.9);
	    \filldraw[green!30, draw=black] (57.6,4.9) rectangle (62.5,9.8);
        \draw[draw = red] (57.6, 0) rectangle (62.5,14.4);
	    \fill (58.68, 4.80) circle (5pt);
	    \fill (59.63, 1.60) circle (5pt);
	    \fill (60.12, 4.46) circle (5pt);
	    \fill (61.67, 1.55) circle (5pt);
	    \fill (59.43, 1.98) circle (5pt);
	    \fill (62.22, 1.49) circle (5pt);
	    \fill (57.93, 3.53) circle (5pt);
	    \fill (60.61, 0.40) circle (5pt);
	    \fill (60.66, 0.58) circle (5pt);
	    \fill (57.87, 3.05) circle (5pt);
	    \fill (60.56, 7.40) circle (5pt);
	    \fill (62.19, 8.07) circle (5pt);
	    \fill (62.28, 9.12) circle (5pt);
	    \fill (61.59, 9.58) circle (5pt);
	    \fill (59.06, 7.47) circle (5pt);
	    \fill (57.64, 8.11) circle (5pt);
	    \fill (58.10, 8.08) circle (5pt);
	    \fill (62.28, 6.57) circle (5pt);
	    \fill (59.59, 8.90) circle (5pt);
	    \fill (60.51, 7.98) circle (5pt);

	    \filldraw[blue!30, draw=black] (78.4,0) rectangle (84.1,5.7);
	    \filldraw[green!30, draw=black] (78.4,5.7) rectangle (84.1,11.4);
	   	\filldraw[orange!30, draw=black] (78.4,11.4) rectangle (84.1,17.1);
        \draw[draw = red] (78.4, 0) rectangle (84.1, 19.6);
	   	\draw[thick] (78.4,-1) -- (78.4,1);
	   	\draw[thick] (84.1,-1) -- (84.1,1);
	   	\node at (78.4, -4) {$28^2$};
	   	\node at (85, -6) {$29^2$};\fill (79.47, 4.55) circle (5pt);
\fill (79.70, 2.77) circle (5pt);
\fill (83.76, 0.06) circle (5pt);
\fill (83.61, 0.54) circle (5pt);
\fill (79.65, 5.22) circle (5pt);
\fill (79.72, 1.26) circle (5pt);
\fill (80.32, 4.80) circle (5pt);
\fill (80.01, 0.77) circle (5pt);
\fill (79.44, 0.05) circle (5pt);
\fill (80.33, 5.58) circle (5pt);
\fill (82.73, 4.06) circle (5pt);
\fill (80.89, 1.04) circle (5pt);
\fill (83.99, 4.93) circle (5pt);
\fill (80.36, 2.66) circle (5pt);
\fill (79.53, 5.09) circle (5pt);
\fill (82.09, 1.26) circle (5pt);
\fill (81.22, 3.97) circle (5pt);
\fill (83.72, 0.97) circle (5pt);
\fill (82.04, 4.93) circle (5pt);
\fill (80.31, 2.64) circle (5pt);
\fill (82.12, 1.56) circle (5pt);
\fill (80.76, 10.38) circle (5pt);
\fill (80.63, 5.84) circle (5pt);
\fill (78.78, 9.08) circle (5pt);
\fill (81.84, 11.07) circle (5pt);
\fill (81.60, 6.91) circle (5pt);
\fill (82.29, 7.07) circle (5pt);
\fill (82.15, 6.72) circle (5pt);
\fill (81.85, 6.58) circle (5pt);
\fill (80.81, 6.93) circle (5pt);
\fill (79.37, 8.05) circle (5pt);
\fill (83.57, 6.24) circle (5pt);
\fill (79.70, 6.87) circle (5pt);
\fill (84.08, 5.87) circle (5pt);
\fill (81.99, 11.05) circle (5pt);
\fill (80.73, 7.78) circle (5pt);
\fill (79.27, 7.21) circle (5pt);
\fill (81.50, 16.10) circle (5pt);
\fill (83.94, 12.11) circle (5pt);
\fill (80.12, 16.81) circle (5pt);
\fill (82.30, 11.73) circle (5pt);
\fill (83.02, 14.11) circle (5pt);
\fill (79.39, 16.40) circle (5pt);
\fill (79.08, 14.13) circle (5pt);
\fill (82.38, 15.64) circle (5pt);
\fill (82.39, 13.80) circle (5pt);
\fill (83.66, 13.26) circle (5pt);
\fill (81.54, 14.75) circle (5pt);
\fill (83.72, 14.23) circle (5pt);
\fill (79.05, 15.19) circle (5pt);
\fill (81.16, 15.71) circle (5pt);
\fill (83.55, 12.06) circle (5pt);
\fill (81.80, 12.74) circle (5pt);
\fill (82.34, 14.78) circle (5pt);
\fill (80.39, 14.73) circle (5pt);
\fill (81.92, 16.03) circle (5pt);
\fill (80.40, 16.50) circle (5pt);
\end{tikzpicture}
\caption{Point set $S$ if $k(n) = n^{0.5}$, with $m\ge 4$}
\label{nalpha}
\end{figure}
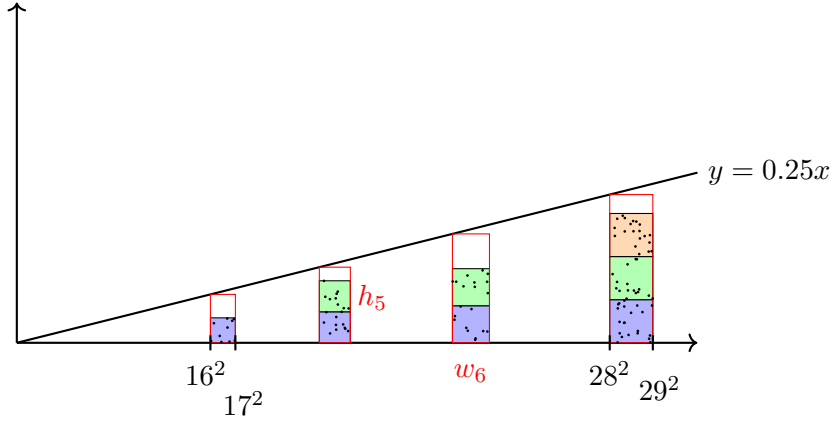

    \subsubsection{Proving the density}

    To compute the density, we need to estimate the size of the squares.
    We approximate $w_m$ with Taylor-series with the Lagrange remainder, i.e. for some $\varepsilon\in(0, 1)$:
    $$
    w_m = c^{-\frac1\alpha}\left(\frac1\alpha (Am)^{\frac1\alpha -1}+ \frac12 \frac1\alpha \left(\frac{1}{\alpha }-1\right)(Am + \varepsilon)^{\frac1\alpha -2}\right)
    $$
    If $m$ is large enough, the following inequality holds:
    $$
    \frac12 \frac1\alpha \left(\frac{1}{\alpha }-1\right)(Am + \varepsilon)^{\frac{1}{\alpha}-2} \le \frac{1}{2\alpha^2}(Am)^{\frac{1}{\alpha}-2}
    $$
    Rearranging the terms yields
    $$
    1- \alpha \le \left(\frac{Am}{Am+\varepsilon}\right)^{\frac1\alpha - 2},
    $$
    Which holds if $m$ is large enough.

    Thus we get the following upper and lower bounds:
    $$
    c^{-\frac1\alpha}\frac1\alpha (Am)^{\frac1\alpha -1}\le w_m \le c^{-\frac1\alpha}\frac1\alpha (Am)^{\frac1\alpha -1} + c^{-\frac1\alpha}\frac{1}{2\alpha^2}(Am)^{\frac{1}{\alpha} -2 }
    $$
    Approximating $s_m$:
    $$
    s_m \ge \frac{h_m}{w_m} -1 \ge \frac{\frac{1}{4}(Am)^{\frac1\alpha}}{\frac1\alpha (Am)^{\frac1\alpha -1} + \frac{1}{2\alpha^2}(Am)^{\frac{1}{\alpha} -2 }}-1 =  \frac{\frac{1}{4} + \frac\alpha4 \frac{1}{2\alpha^2}(Am)^{ -1 } - \frac\alpha4 \frac{1}{2\alpha^2}(Am)^{-1 }}{\frac1\alpha (Am)^{ -1} + \frac{1}{2\alpha^2}(Am)^{-2 }}-1 =$$ 
    $$
    = \frac14 \alpha Am -\frac{ \frac\alpha4 \frac{1}{2\alpha^2}(Am)^{-1 }}{\frac1\alpha (Am)^{ -1} +  \frac{1}{2\alpha^2}(Am)^{-2 }}-1  = \frac14 \alpha Am -\frac{\frac\alpha4}{2\alpha  + (Am)^{-1}}-1\ge 
    \frac14 \alpha Am - \frac{1}{8} -1 
    \ge 
    \frac14\alpha Am-2$$
    Also:
    $$
    s_m\le \frac{h_m}{w_m}\le \frac14\alpha Am
    $$
    Thus:
    $$
    \frac14 \alpha Am - 2\le s_m\le \frac14 \alpha Am
    $$   
    Let $M_0$ be the index of the first tower containing a positive number of points. Let $M_n$ be the index of the last tower, that is entirely inside $[1, n]^2$. 
    $$
    \left(\frac{AM_n + 1}{c}\right)^\frac{1}{\alpha} \le n
    $$
    $$
    M_n = \left \lfloor\frac{cn^\alpha - 1}{A}\right \rfloor
    $$   
    In a general $[1,n]^2$ square the number of points is
    $$
    |S \cap [1,n]^2| 
    \ge \sum_{M_0\le m\le M_n} w_m \cdot s_m
    \ge \sum_{M_0\le m \le M_n} \left(c^{-\frac1\alpha}\frac1\alpha (Am)^{\frac1\alpha -1}\right) \cdot \left(\frac14 \alpha Am - 2\right)
    $$
    Thus, for the constant $L = 8c^{-\frac1\alpha} \frac1\alpha A^{\frac1\alpha -1}$:
    
    $$
    |S \cap [1,n]^2| \ge   \frac14 \sum_{M_0\le m\le M_n}c^{-\frac1\alpha} (Am)^{\frac{1}{\alpha}} - Lm^{\frac1\alpha -1} = \frac14 \sum_{M_0\le m\le M_n} g(m)
    $$
    Where 
    $$g(m) = c^{-\frac1\alpha} (Am)^{\frac{1}{\alpha}} - Lm^{\frac1\alpha -1} = m^{\frac1\alpha -1}(c^{-\frac1\alpha}A^{\frac1\alpha}m-L )$$
    Note that $g(m)$ is a monotone increasing function if $m \ge L\cdot (c/A)^{1/\alpha}$. Choose $M_0$ accordingly. Nonetheless, we can rewrite the sum as an integral with $M_n = { \lfloor\frac{cn^\alpha - 1}{A}\rfloor}$:
    $$
    |S \cap [1,n]^2|\geq \frac14 \int_{M_0-1}^{M_n} g(x) dx
    = \frac14 \int_0^{\frac{cn^\alpha}{A}} g(x)dx -\frac14\int_0^{M_0-1}g(x) dx - \frac14 \int_{ \lfloor\frac{cn^\alpha - 1}{A}\rfloor}^{\frac{cn^\alpha}{A}} g(x)dx
    $$
    $$
   |S \cap [1,n]^2| \ge \left[\frac14 \frac cA \frac{\alpha}{\alpha+1}n^{\alpha + 1}  - Dn\right] -  E- [Fn - G]
    $$
    for some constants $D, E, F, G$. Now examining the $\liminf$:
    $$
    \liminf \frac{|S \cap [1,n]^2|}{n \cdot cn^\alpha}
    \geq \liminf \frac{\frac14 \frac cA \frac{\alpha}{\alpha+1}n^{\alpha + 1}  - (D+F)n  - (E+G)}{n \cdot cn^\alpha}$$
    By the choice of $A=4$, this limit equals to $$\frac{1}{16} \frac{\alpha}{\alpha + 1}
    $$
    which is indeed positive.
    \subsubsection{Proving that $S$ is a no-$\mathbf{(k(n)+1)}$-in-line set}
    
    We now prove that the construction satisfies the no-$k(n)+1$-in-line property. We check that every line has less than $cn^\alpha$ points for every $n$. Denote the index of the last tower which intersects $[1, n]^2$ by $M_n'$
    $$
    \left(\frac{AM_n'}{c}\right)^\frac{1}{\alpha} \le n
    $$
    $$
    M_n' = \left\lfloor\frac{cn^\alpha}{A}\right\rfloor\le \frac{cn^\alpha}{A}
    $$
    Denote the number of towers in $[1, n]^2$ by $T_n$. Since we omitted the first couple of towers of $S$, $T_n = M_n' - M_0' + 1$, where $M_0'>1$ is the index of the first non-empty tower. Thus:
    $$
    T_n <  \frac{cn^\alpha}{A}
    $$
    Denote the number of points on the examined line by $\#l \in \mathbb N$. We always need $\#l \le cn^\alpha$. Let the examined line be in the form of $l: y = q(x - r)$.

    \subparagraph{\bf I. Vertical lines: }
    A vertical line intersects at most one tower and a single tower contains a maximum of $s_{M_n'} \le \frac14\alpha cn^\alpha$ squares. Since each square contains points in general position, a vertical line can contain at most
    $$\#l \le 2\cdot s_{M_n'} \le \frac12\alpha cn^\alpha <cn^\alpha $$
    points.
    
    \subparagraph{\bf II. $\mathbf{-2 < q<2}$: }
    In this case, the line intersects at most two squares in each tower, thus the maximal possible number of points on the line is
	$$\# l \le 2\cdot 2T_n < 2\cdot 2\cdot cn^\alpha/A  =  cn^\alpha$$

    \subparagraph{\bf III./a) $\mathbf{2\le q}$ and the line intersects at most two towers: }
    The line intersects at most $2\cdot s_{M_n'}$ squares. The upper bound is even stronger in this case:
    $$\#l \le 2\cdot 2\cdot s_{M_n'}\le cn^\alpha$$

    \subparagraph{\bf III./b) $\mathbf{2\le q}$ and the line intersects three or more towers: } \label{haromb}
	We calculate the number of towers between $r$, and the point where $l$ intersects $y = \frac14x$. This intersection is located at $x = \frac{qr}{q- \frac14}$ as seen in Figure \ref{IIIbq2}.
    If there are $T_x$ towers till $x$ and $T_r$ till $r$, then the number of towers the line intersects is $ T_x - T_r + 1$:
    $$
    3\le T_x - T_r + 1\le \left\lfloor\frac{c\left(\frac{qr}{q- \frac14}\right)^\alpha }{A}\right\rfloor - \left\lfloor \frac{cr^\alpha}{A}\right\rfloor +1 \le \frac{c\left(\frac{qr}{q- \frac14}\right)^\alpha}{A} - \frac{cr^\alpha}{A} + 2 \le 3\cdot \left(\frac{c\left(\frac{qr}{q- \frac14}\right)^\alpha}{A} - \frac{cr^\alpha}{A}\right)
    $$
	In a single tower, the line intersects at most $q+1 \le \frac32q$ squares, so it intersects at most $3q$ points in each square. Therefore, the number of points on $l$ is at most:
    $$
    \#l \le 3\cdot\left(\frac{c\left(\frac{qr}{q- \frac14}\right)^\alpha}{A} - \frac{cr^\alpha}{A}\right)3q= \frac{ 9r^\alpha c}{A}\left(\left(\frac{q}{q-\frac14}\right)^\alpha - 1\right)q$$
    Using $x^\alpha -1 \le \alpha(x-1)$ if $x\ge 1$:
    $$
    \left(\left(\frac{q}{q-\frac14}\right)^\alpha - 1\right)q \le \alpha q \left(\frac{q}{q - \frac14 } - 1 \right) = \alpha\frac{ 1}{4 -\frac1q} \le \frac27 \alpha
    $$
    Since we can assume $r\le n$, we have
$$\# l \le \frac{18}{7A}cn^\alpha < cn^\alpha$$

\subparagraph{\bf IV.  $\mathbf{q\le -2}$:}
    Denote the line by $l$ as shown in Figure \ref{IVq-2}. Let $X_1$ be the intersection point of $l$ and the $x$-axis, and $X_2$ be the intersection of $l$ and $y = 1/4 x$. Denote the midpoint of segment $X_1X_2$ by $X$, and denote the reflection of $l$ to the horizontal line passing through $X$ by $l'$. Obviously for the slope $q'$ of $l'$ we have $q' = -q$. Every tower intersected by $l$ is intersected by $l'$ as well, therefore the same upper bound as in Case III./b works in this case as well.

\begin{figure}[htbp]
    \centering
    \begin{minipage}{0.48\textwidth}
        \centering
        \begin{tikzpicture}[scale = 0.065]
            \fill (40, 10) circle (15pt);
            \fill (60, 0) circle (15pt);
            \fill (50, 5) circle (15pt);
            \node at (40, 15) {$X_2$};
            \draw[dashed] (40, 10) -- (40, 0);
            \node at (60, -4) {$X_1$};
            \node at (57, 5) {$X$};
            \draw[->, thick] (0,0) -- (90,0);
            \draw[->, thick] (0,0) -- (0,70);
            \draw[thick, domain=0:90] plot(\x, {0.25*\x}) node[left, xshift= -25pt] {$y = 0.25x$};
            \draw[thick, domain=8:70] plot(\x, {-0.5*(\x -60)}) node[right] {$l$};
            \draw[thick, domain=25:90] plot(\x, {0.5*(\x - 40)}) node[above] {$l'$};
        \end{tikzpicture}
        \caption{IV. $q\le -2$}
        \label{IVq-2}
    \end{minipage}\hfill
    \begin{minipage}{0.48\textwidth}
        \centering
        \begin{tikzpicture}[scale = 0.065]
            \fill (50, 0) circle (15pt);
            \fill (60, 0) circle (15pt);
            \node at (50, -4) {$r$};
            \node at (60, -4) {$x$};
            \draw[->, thick] (0,0) -- (90,0);
            \draw[->, thick] (0,0) -- (0,70);
            \draw[thick, domain=0:90] plot(\x, {0.25*\x}) node[above] {$y = 0.25x$};
            \draw[thick, domain=45:70] plot(\x, {1.5*(\x -50)}) node[right] {$l$};
            \draw[draw=red] (49.5,0) rectangle (50.5,12.375);
            \draw[draw=red] (52.5,0) rectangle (54,13.125);
            \draw[draw=red] (55.5,0) rectangle (57.5,13.875);				
            \draw[draw=red] (59.5,0) rectangle (62,14.875);
        \end{tikzpicture}
        \caption{III/b) $q\ge 2$ }
        \label{IIIbq2}
    \end{minipage}
\end{figure}

\medskip
\noindent
This concludes the proof of Theorem \ref{posliminf_linear}.
\begin{remark}
    We did not attempt to optimize the density of the $\liminf$ with this type of construction, it might be larger, we suspect that in fact, this construction is not optimal. Furthermore, we checked that this 'tower' construction works for polynomially growing functions $k(n)$ only, and does not provide a positive density set for slower growing functions.
\end{remark}
\end{proof}

\section{Criteria for large $\displaystyle\liminf_{n\to\infty}\frac{|S_n|}{k(n)n}$}\label{criteria}

We begin this section by proving two theorems about the relationship of $\liminf \frac{|S_n|}{k(n)\cdot n}$ and the growth rate of $k(n)$.

\begin{theorem}
Let $k:\nn^+\to\mathbb{R}^+$ be a monotone increasing function and $S\subseteq\zz^{2}$ a set of lattice points such that $\left|S\cap\ell\cap[1,n]^{2}\right|\le k(n)$ for all lines $\ell$. If $S_n:=S\cap [1,n]^2$ satisfies
$$\liminf_{n\to\infty} \frac{|S_n|}{k(n)\cdot n} \ge 0.897,$$
then $k(n) = \Omega( n^c)$, where $c$ is a positive absolute constant.
\end{theorem}

\begin{proof}
By the condition on $S_n$, there is a constant $n_0$ satisfying $\frac{|S \cap [1,n]^2|}{n k(n)} >1-\varepsilon$ for every $n \ge n_0$, where $\varepsilon=0.1036$. Pick $n_0$ such that it further satisfies $k(n_0) > 0$. Subdivide $[1, mn_0]^2$ into $m\times m$ tiles of size $n_0\times n_0$. 
 
 By the definition of $n_0$, there are at least $(1-\varepsilon)(m-1)n_0k((m-1)n_0)$ points from $S$ in $[1,(m-1)n_0]^2$. $[1,(m-1)n_0]\times[1,mn_0]$ contains $(m-1)n_0$ disjoint vertical lines inside the  square $[1, mn_0]^2$. Consequently, it contains at most $(m-1)n_0k(mn_0)$ points of $S$.
  Therefore, the rectangle $[1,(m-1)n_0]\times[(m-1)n_0+1,mn_0]$ contains at most $$(m-1)n_0k(mn_0)-(1-\varepsilon)(m-1)n_0k((m-1)n_0)$$ points of $S$.  
  
  The rectangle $[1,mn_0]\times[(m-1)n_0+1,mn_0]$ is the disjoint union of $n_0$ horizontal lines inside the square $[1, mn_0]^2$. The other $(m-1)n_0$ horizontal lines contain at most $(m-1)n_0k(mn_0)$ points inside the square $[1, mn_0]^2$. By the definition of $n_0$, the square $[1, mn_0]^2$ contains at least $(1-\varepsilon)mn_0k(mn_0)$ points from $S$. Therefore the rectangle $[1,mn_0]\times[(m-1)n_0+1,mn_0]$ contains at least $(1-m\varepsilon)n_0k(mn_0)$ points from $S$. Together with the previous bound, this implies that $[(m-1)n+1,mn]^2$ contains at least 

  $$   (1-m\varepsilon)n_0k(mn_0)-\big((m-1)n_0k(mn_0)-(1-\varepsilon)(m-1)n_0k((m-1)n_0)\big)=$$ 
 $$=((m-1-(m-1)\varepsilon)k((m-1)n_0)-(m-2+m\varepsilon)k(mn_0))n_0$$
 points of $S$. This bound applies to all squares in the main diagonal of $[1, mn_0]^2$.

The $m$ squares of size $n_0\times n_0$ in the main diagonal of $[1,mn_0]^2$ can be covered by $2n_0-1$ diagonal lines of slope 1, inside the square $[1,mn_0]^2$, which means they contain at most $(2n_0-1)k(mn_0)\le2n_0k(mn_0)$ points from $S$. Comparing this bound with the previous lower bound on these squares yields 
$$\sum_{i=1}^{m}((i-1-(i-1)\varepsilon)k((i-1)n_0)-(i-2+i\varepsilon)k(in_0))n_0\le2n_0k(mn_0)$$
$$\sum_{i=1}^{m-1}(2-2i\varepsilon)k(in_0)-(m-2+m\varepsilon)k(mn_0)\le 2k(mn_0)$$
\begin{equation}
\sum_{i=1}^{m-1}2(1-i\varepsilon)k(in_0)\le m(1+\varepsilon)k(mn_0).
\label{eq:sum_bound}
\end{equation}

In the following we recursively establish an inequality between $k(n_0)$ and $k(5n_0)$.

First apply inequality \ref{eq:sum_bound} for $m=2$:

$$2(1-\varepsilon)k(n_0)\le2(1+\varepsilon)k(2n_0)$$
$$\frac{1-\varepsilon}{1+\varepsilon}k(n_0)\le k(2n_0).$$

Proceeding with $m=3$ yields

$$2(1-\varepsilon)k(n_0)+2(1-2\varepsilon)k(2n_0)\le3(1+\varepsilon)k(3n_0).$$
Using $\frac{1-\varepsilon}{1+\varepsilon}k(n_0)\le k(2n_0)$:
$$\left((1-\varepsilon)+(1-2\varepsilon)\frac{1-\varepsilon}{1+\varepsilon}\right)2k(n_0)\le 3(1+\varepsilon)k(3n_0)$$
$$\frac{2(1-\varepsilon)(2-\varepsilon)}{1+\varepsilon}k(n_0)\le 3(1+\varepsilon)k(3n_0)$$
$$\frac{2(1-\varepsilon)(2-\varepsilon)}{3(1+\varepsilon)^2}k(n_0)\le k(3n_0).$$
Applying inequality \ref{eq:sum_bound} for $m=4$ yields
$$2(1-\varepsilon)k(n_0)+2(1-2\varepsilon)k(2n_0)+2(1-3\varepsilon)k(3n_0)\le4(1+\varepsilon)k(4n_0)$$
$$\left(\frac{2(1-\varepsilon)(2-\varepsilon)}{1+\varepsilon}+2(1-3\varepsilon)\frac{2(1-\varepsilon)(2-\varepsilon)}{3(1+\varepsilon)^2}\right)k(n_0)\le4(1+\varepsilon)k(4n_0)$$
$$\frac{2(1-\varepsilon)(2-\varepsilon)(5-3\varepsilon)}{3(1+\varepsilon)^2}k(n_0)\le 4(1+\varepsilon)k(4n_0)$$
$$\frac{(1-\varepsilon)(2-\varepsilon)(5-3\varepsilon)}{6(1+\varepsilon)^3}k(n_0)\le k(4n_0).$$
For $m=5$ the process yields
$$2(1-\varepsilon)k(n_0)+2(1-2\varepsilon)k(2n_0)+2(1-3\varepsilon)k(3n_0)+2(1-4\varepsilon)k(4n_0)\le5(1+\varepsilon)k(5n_0)$$
$$\left(\frac{2(1-\varepsilon)(2-\varepsilon)(5-3\varepsilon)}{3(1+\varepsilon)^2}+2(1-4\varepsilon)\frac{(1-\varepsilon)(2-\varepsilon)(5-3\varepsilon)}{6(1+\varepsilon)^3}\right)k(n_0)\le 5(1+\varepsilon)k(5n_0)$$
$$\frac{(1-\varepsilon)(2-\varepsilon)(5-3\varepsilon)(3-2\varepsilon)}{15(1+\varepsilon)^4}\le k(5n_0).$$
$\varepsilon=0.1036$ means $Ck(n_0)\le k(6n_0)$, where $C=1.00054>1$ is a constant. We can repeat this argument for every $n\ge n_0$, including every $n=5^an_0$, where $a\in\nn$, this result gives us 
$$C^ak(n_0)\le k(5^an_0)=k(n)$$
$$C^{\log_5\frac{n}{n_0}}k(n_0)=5^{\log_5C\log_5\frac{n}{n_0}}k(n_0)=\left(\frac{n}{n_0}\right)^{\log_5C}k(n_0)\le k(n).$$
Setting $c=\log_5C\approx 0.0003>0$ and the monotonicity of $k(n)$ means $k(n)=\Omega(n^c)$.
    
\end{proof}

We fitted the value of $\varepsilon$ for $m=5$ as it gives the highest possible value, meaning the best bound in our theorem.

This theorem proves that we cannot achieve $\liminf\frac{|S_n|}{k(n) n}>0.897$ for functions such as $k(n)=c\log n$ or when $k(n)$ is constant. Experimental evidence of Nagy, Nagy and Woodroofe \cite{nagy2023extensible} suggests that at the extensible no-three-in-line problem, which is the case of $k(n)=2$, there exists a set $S$ without three collinear points, such that $S\cap[1,n]^2$ is slightly larger than $0.8n$. However, a significant gap remains between these bounds as the computational estimate for $\liminf\frac{|S_n|}{k(n) n}=\liminf\frac{|S_n|}{2 n}$ is approximately $0.4$, while we verified that it is less than $0.897$.

Setting $m>5$ in the previous proof would require a smaller $\varepsilon$ value, because the error terms in the main diagonal would pile up. Setting $\liminf\frac{|S_n|}{k(n) n}=1$ will mean that we can choose $\varepsilon>0$ to be arbitrarily small. In that case, we can repeat the previous proof for fixed $m\in \nn$ by choosing a suitable $\varepsilon>0$ to have insignificant error. In the limit $m\to\infty$ this gives us the following result.

\begin{theorem}
Let $k:\nn^+\to\mathbb{R}^+$ be a monotone increasing function. Suppose there exists a set of grid points $S\subseteq\zz^{2}$, such that $\left|S\cap\ell\cap[1,n]^{2}\right|\le k(n)$ for all lines $\ell$. If $S_n:=S\cap [1,n]^2$ satisfies
$$\liminf_{n\to\infty} \frac{|S_n|}{k(n)\cdot n} = 1,$$
then $k(n) = \Omega( n^c)$, for every constant $c<1$.
\end{theorem}

\begin{proof}

We fix $m\in\nn$ and choose an $\varepsilon>0$, which we will specify later. We can find an $n_0$ that guarantees for every $n \ge n_0: \frac{|S \cap [1,n]^2|}{n k(n)} >1-\varepsilon$ and pick such an $n_0$, that satisfies $k(n_0) > 0$.  We subdivide the square $[1, mn_0]^2$ into $m\times m$ squares of size $n_0\times n_0$. As verified in the previous proof, 
$$\sum_{i=1}^{m-1}(2-2i\varepsilon)k(in_0)\le (m+m\varepsilon)k(mn_0).$$
The same formula holds for any $0<m'<m$.
Applying this recursion, we show that $$k(mn_0)\ge\frac{m+1}{3}k(n_0)\frac{1+\varepsilon p_m(\varepsilon)}{1+\varepsilon q_m(\varepsilon)},$$ where $p_m(x),q_m(x)$  are polynomials with real coefficients. We will use this to prove formally that we can choose $\varepsilon>0$ small enough, such that $\frac{m+1}{3}\frac{1+\varepsilon p_1(\varepsilon)}{1+\varepsilon q_1(\varepsilon)}> \frac{m+1}{3}$. Notice that $\exists p_3(x),q_3(x)\in \mathbb{R}[x]$ 
$$a\frac{1+\varepsilon p_1(\varepsilon)}{1+\varepsilon q_1(\varepsilon)}+b\frac{1+\varepsilon p_2(\varepsilon)}{1+\varepsilon q_2(\varepsilon)}=(a+b)\frac{1+\varepsilon p_3(\varepsilon)}{1+\varepsilon q_3(\varepsilon)},$$
 where $p_1(x),q_1(x),p_2(x),q_2(x)\in \mathbb{R}[x]$ arbitrary polynomials and $a+b\ne0$. Similarly  $\exists p_3(x),q_3(x)\in \mathbb{R}[x]$ 
 $$a\frac{1+\varepsilon p_1(\varepsilon)}{1+\varepsilon q_1(\varepsilon)}\cdot b\frac{1+\varepsilon p_2(\varepsilon)}{1+\varepsilon q_2(\varepsilon)}=ab\cdot\frac{1+\varepsilon p_3(\varepsilon)}{1+\varepsilon q_3(\varepsilon)},$$ 
 where $p_1(x),q_1(x),p_2(x),q_2(x)\in \mathbb{R}[x]$ arbitrary polynomials. Similarly  $\exists p_3(x),q_3(x)\in \mathbb{R}[x]$ 
 $$a\frac{1+\varepsilon p_1(\varepsilon)}{1+\varepsilon q_1(\varepsilon)}/\left(b\frac{1+\varepsilon p_2(\varepsilon)}{1+\varepsilon q_2(\varepsilon)}\right)=\frac ab\cdot \frac{1+\varepsilon p_3(\varepsilon)}{1+\varepsilon q_3(\varepsilon)},$$ 
 where $p_1(x),q_1(x),p_2(x),q_2(x)\in \mathbb{R}[x]$ arbitrary polynomials and $b\ne0$; $a,b\in \mathbb{R}$ constants in every expression.

 These properties mean that solving the recursion without $\varepsilon$-s and not adding constants that sum to $0$ is enough to prove the recursion.

 It is enough to prove $\sum_{i=1}^{m-1}2k(in_0)\le mk(mn_0)\Rightarrow k(mn_0)\ge\frac{m+1}{3}k(n_0)$. We prove this by induction, if every $i<m$ satisfies this condition, we have 
 $$mk(mn_0)\ge\sum_{i=1}^{m-1}2k(in_0)\ge 2k(n_0)\left(1+\sum_{i=2}^{m-1}\frac{m+1}{3}\right)=2k(n_0)\left(1+\frac{(m-2)(m+3)}{6}\right)$$
 $$mk(mn_0)\ge\frac{m^2+m}{3}k(n_0)$$
 $$k(mn_0)\ge\frac{m+1}{3}k(n_0).$$
 $k(n_0)>0$, thus we never summed constants adding up to $0$ and this proof works with multiplying every constant with a suitable $\frac{1+\varepsilon p(\varepsilon)}{1+\varepsilon q(\varepsilon)}$, where $p(x),q(x)\in \mathbb{R}[x]$ are polynomials.

 This proves $k(mn_0)\ge\frac{m+1}{3}k(n_0)\frac{1+\varepsilon p_m(\varepsilon)}{1+\varepsilon q_m(\varepsilon)}$ and we can choose $\varepsilon>0$ to be small enough that $k(mn_0)\ge\frac{m+1}{3}k(n_0)\frac{1+\varepsilon p_m(\varepsilon)}{1+\varepsilon q_m(\varepsilon)}\ge\frac{m}{3}k(n_0)$.

 We have chosen $n_0$, such that we can repeat this argument for every $n\ge n_0$, including every $n=m^an_0$, where $a\in\nn$, this result gives us 
$$k(n)= k(m^an_0)\ge\frac{m^a}{3^a}k(n_0)=\frac{n/n_0}{3^a}k(n_0)$$
$$k(n)\ge\frac{n}{n_0}3^{-log_m(n/n_0)}k(n_0)=\frac{n}{n_0}m^{-\log_m(n/n_0)\log_m3}k(n_0)=\left(\frac{n}{n_0}\right)^{1-\log_m3}k(n_0).$$
We can choose $m\in  \nn$ as large as we want it to be, making $\log_m3>0$, arbitrarily close to 0, thus $k(n)=\Omega(n^c)$, where $c$ can be any constant smaller than $1$.

\end{proof}

The above theorem proves that there is no $S$ with $\liminf\frac{|S\cap[1,n]^2|}{k(n) n}=1$ for example for functions $k(n)=c n^{\alpha}$, where $\alpha<1$. In the following, we provide further estimations that are applicable for $k(n)= O(n)$. We will use these bounds to prove that for any constant $c\in[0,1]$ there exists $k(n)$ and $S$ such that $\liminf \frac{|S \cap [1,n]^2|}{n k(n)}=c$ and this is the highest possible for this $k(n)$.

\begin{lemma}\label{upperboundlemma}
For any given no-$(k(n)+1)$-in-line set $S$:
$$\liminf_{n\to\infty} \frac{|S \cap [1,n]^2|}{n k(n)} \leq \liminf_{n\to\infty} \frac{k(n)}{k(n+1)}.$$
\end{lemma}
\begin{proof}
For $n>1$ consider the $(n+1)$th row and column separately from the rest of the points of $[1,n+1]^2$. In the first $[1, n]^2$ there are at most $nk(n)$ points. In the $(n+1)$th row and column there are at most $k(n+1)$ points each. So altogether there are at most $nk(n)+2k(n+1)$ points in $[1,n+1]^2$. This means the ratio
$$\liminf_{n\to\infty} \frac{|S \cap [1,n]^2|}{n k(n)}=\liminf_{n\to\infty} \frac{|S \cap [1,n+1]^2|}{(n+1) k(n+1)} \leq\liminf_{n\to\infty} \frac{nk(n)+2k(n+1)}{(n+1)k(n+1)}=\liminf_{n\to\infty} \frac{k(n)}{k(n+1)}.$$
\end{proof}
This is a special case of the following lemma ($f(n)=1$ case).
\begin{lemma}
For any given no-$(k(n)+1)$-in-line set $S$:
$$\liminf_{n\to\infty} \frac{|S \cap [1,n]^2|}{n k(n)} \leq \liminf_{n\to\infty} \left(\frac{k(n)}{(1+\frac{f(n)}{n})k(n+f(n))}+\frac{2f(n)}{n+f(n)}\right).$$
If $f(n)=o(n)$, this becomes $$\liminf_{n\to\infty} \frac{|S \cap [1,n]^2|}{n k(n)} \leq \liminf_{n\to\infty} \frac{k(n)}{k(n+f(n))}.$$
If $f(n)=(c-1)n$ for $c\in (1, 2)$:
$$\liminf_{n\to\infty} \frac{|S \cap [1,n]^2|}{n k(n)} \leq \liminf_{n\to\infty} \left(\frac{k(n)}{ck(cn)}+\frac{2(c-1)}{c}\right).$$
\end{lemma}
\begin{proof}
    In the first $[1, n]^2$ there are at most $nk(n)$ points. In the next $f(n)$ rows and columns, there are at most $k(n+f(n))$ points each. The $\liminf$ at points $n+f(n)$ gives:
    $$\liminf_{n\to\infty} \frac{|S \cap [1,n]^2|}{n k(n)}
    \leq \liminf_{n\to\infty} \frac{|S \cap [1,n+f(n)]^2|}{(n+f(n))(k(n+f(n))}\le $$ $$\le\liminf_{n\to\infty}{\frac{nk(n)+2f(n)k(n+f(n))}{(n+f(n))k(n+f(n))}}
    = \liminf_{n\to\infty} \left(\frac{k(n)}{(1+\frac{f(n)}{n})k(n+f(n))}+\frac{2f(n)}{n+f(n)}\right)$$
\end{proof}

\begin{theorem}\label{theorem35}
    For any constant $c\in[0,1]$ we can construct $S$ and $k(n)\le n$ such that $\liminf \frac{|S \cap [1,n]^2|}{n k(n)}=c$ and this is the highest possible value of the $\liminf\frac{|S \cap [1,n]^2|}{n k(n)}$ for this $k(n)$.
\end{theorem}

\begin{proof}
We have seen that linear $k(n)$ are examples for $\displaystyle\liminf_{n\to\infty} \frac{|S \cap [1,n]^2|}{n k(n)}=1$.

The case of $\displaystyle\liminf_{n\to\infty} \frac{|S \cap [1,n]^2|}{n k(n)}=0$ is achieved by $k(n)=m!$ if $m!\le n<(m+1)! $, where $m\in\nn$. This means $\frac{k(m!-1)}{k(m!)}=\frac{(m-1)!}{m!}=\frac1m$, which goes to $0$ as $m\to\infty$. Lemma~\ref{upperboundlemma} proves that $\displaystyle\liminf_{n\to\infty} \frac{|S \cap [1,n]^2|}{n k(n)}=0$.

For any given $c \in (0,1)$, consider the function $k(n) = \lfloor \left(\frac{1}{c}\right)^m\rfloor $ where $n \in\left[\left(\frac{1}{c}\right)^m, (\frac{1}{c})^{m+1} \right)$ for the suitable $m \in \nn^+$. This means that there are jumps in the function at the boundary of these intervals when $\frac{k(n)}{k(n+1)}\approx\frac{(1/c)^m}{(1/c)^{m+1}}=c$ and Lemma~\ref{upperboundlemma} proves that $\displaystyle\liminf_{n\to\infty} \frac{|S \cap [1,n]^2|}{n k(n)}\le c$. Similarly $\displaystyle\liminf_{n\to\infty} \frac{k(n)}{n}\ge c$ and by Lemma~\ref{lowerboundlemma}, this proves $\displaystyle\liminf_{n\to\infty} \frac{|S \cap [1,n]^2|}{n k(n)}\ge c$. The two bounds together give $\displaystyle\liminf_{n\to\infty} \frac{|S \cap [1,n]^2|}{n k(n)}=c$.
\end{proof}
\section{Extending the solution for the extensible no-3-in-Line problem to sublinear limiters}

In this section, we will prove that if there exists a point-set $S$ that has a positive density for the no-3-in-line problem, then for every sufficiently regular function $k(n)$ including $k(n)=cn^\alpha$ and $k(n)=c\log n$, there exists a point set $S_{k(n)}$ that has a positive density for the no-$k(n)+1$-in-line problem.

First, we define a family of functions whose members have suitable properties to be used as extensible limiter functions and have not been discussed in the previous chapters.

\begin{definition}[Sublinear limiter] \label{def:sublinear}
We call $k(n)$ a sublinear limiter function if it satisfies the following properties:
\begin{enumerate}[label=(\arabic*)]
    \item $k(n) = o(n)$ \label{sublin2}
    
    \item $k$ is unbounded \label{sublin3}
    
    \item $k$ is monotonically increasing \label{sublin4}

    \item $\displaystyle\liminf_{n \to \infty} \frac{\sum_{i=1}^{n} 2k(i)}{ n k(n)} \geq 1$ \label{sublin1}

    \item $\displaystyle\liminf_{n \to \infty} \frac{k(\lfloor \lambda n\rfloor)}{k(n)} \geq \lambda$ holds for all $\lambda \in [0,1]$ \label{sublin5}
    
\end{enumerate}
\end{definition}

\begin{remark} These properties arise naturally when defining sublinear limiter functions, this remark explains shortly why each point of the definition is relevant.
\begin{enumerate}
    \item[\ref{sublin2}] 
    The case $k(n) = cn$ was fully discussed in Section~\ref{positive_density_sets} and that construction gives a positive density solution for every $k(n)=\Theta(n)$. We only consider functions with $k(n)=o(n)$ in this section because our goal is a positive density set for $k(n)$.
    
    \item[\ref{sublin4}] Given a function $k(n): \mathbb{N}^+\to \mathbb{N^+}$, one can construct a monotone $k'(n) = \inf_{i \geq n} {k(i)}$ limiter function. Indeed, if the maximum number of points allowed on a line in $[1,m]^2$ is smaller than in $[1,n]^2$ for some $m>n$, then $k(m)$ is an indirect upper bound on the possible number of points on a line in $[1,n]^2$ as well.

    \item[\ref{sublin3}] If $k(n)$ is bounded by a constant then the problem is essentially the same as $k(n)=k$. A bounded, monotone natural valued function is constant from some point $n_0$ and the studied $\liminf \frac{|S\cap [1,n]^2|}{nk(n)}$ will not be different if we delete the points from $[1,n_0]^2$ as we decreased the numerator by a constant and the denominator goes to infinity. We will construct a positive density set for $k(n)$ from a positive density set for $k(n)=3$ which is also a positive density set for $k(n)=k$. As a consequence, we only consider unbounded functions in this section.

    \item[\ref{sublin1}] If the left hand side is smaller than 1, we immediately get an upper bound for the largest possible $\liminf$, because the numerator on the left hand side counts the maximal possible number of points in the $[1,i]^2 \setminus [1,i-1]^2 $ regions, which add up to the number of total allowed points in $[1,n]^2$. Note that for the proof of Theorem \ref{thm:ulb} the positivity of the left hand side is sufficient.

    \item[\ref{sublin5}] We exclude functions with large jumps including those constructed in the proof of Theorem~\ref{theorem35}. 
\end{enumerate}
\end{remark}

In the following theorem we will establish a direct connection between the maximal achievable densities for the no-3-in-line problem and for the no-$k(n)+1$-in-line problem when $k$ is a sublinear limiter.

\begin{theorem}[Universal lower bound] \label{thm:ulb}
Let $L\subseteq \zz^2$ be a set of lattice points such that it contains no three collinear points. Suppose $$\liminf \frac{| L \cap [1,n]^2|}{2n} = H>0.$$ Then for any sublinear limiter $k(n)$, there exists a set of lattice points $S\subseteq \zz^2$ having no $k(n)+1$ collinear points in $[1,n]^2$, with $$\liminf \frac{|S\cap [1,n]^2|}{nk(n)} \geq H^2/4.$$
\end{theorem}
\begin{proof}
For a given $k$, we construct a set $S$ that contains at most $ k(n)$ points in $[1,n]^2\cap \ell$ for every line $\ell$. 
Let $x_n = \lfloor \frac{k(n)}{2} \rfloor - \lfloor \frac{k(n-1)}{2} \rfloor$ for all $n\in \nn^+$ and let $L_{i,j}$ (for all $1\le i < \infty, 0\le j < x_i$) consist of the points of $L$ shifted by the vector $$v_{i, j}=(0, j+ \sum_{l = 1}^{i - 1}x_l)$$ with no point lying in $[1, i-1]^2$;
$$L_{i, j}:= (L + v_{i, j})\setminus[1, i-1]^2.$$
Let $S$ be the union of these shifted copies of $L$; 
$$S=\bigcup_{\substack{1 \le i  \\ 0 \le j < x_i}} L_{i,j}.$$

\begin{figure}[h!!]
\centering
\begin{tikzpicture}[scale=0.3]

  \draw[step=1cm,lightgray,very thin] (0,0) grid (29,33);

  \draw[->,thick] (0,0) -- (30,0) node[right] {$x$};
  \draw[->,thick] (0,0) -- (0,34) node[above] {$y$};

  \draw[red, thick, dashed] (0,0) rectangle (10,10);

  \node[red] at (5,5.5) {\small\textbf{Deleted region}};

  \foreach \x in {0,...,28} {
    \pgfmathsetmacro\y{mod(\x*\x,29)}
    \pgfmathsetmacro\xnew{28 - \y}
    \pgfmathsetmacro\ynew{\x}
    \ifdim\xnew pt>10pt \filldraw[blue] (\xnew,\ynew) circle (0.3);
    \else\ifdim\ynew pt>10pt \filldraw[blue] (\xnew,\ynew) circle (0.3);\fi\fi
  }

  \foreach \x in {0,...,28} {
    \pgfmathsetmacro\y{mod(\x*\x,29)}
    \pgfmathsetmacro\xnew{28 - \y}
    \pgfmathsetmacro\ynew{\x + 1}
    \ifdim\xnew pt>10pt \filldraw[red] (\xnew,\ynew) circle (0.3);
    \else\ifdim\ynew pt>10pt \filldraw[red] (\xnew,\ynew) circle (0.3);\fi\fi
  }

  \foreach \x in {0,...,28} {
    \pgfmathsetmacro\y{mod(\x*\x,29)}
    \pgfmathsetmacro\xnew{28 - \y}
    \pgfmathsetmacro\ynew{\x + 2}
    \ifdim\xnew pt>10pt \filldraw[green!60!black] (\xnew,\ynew) circle (0.3);
    \else\ifdim\ynew pt>10pt \filldraw[green!60!black] (\xnew,\ynew) circle (0.3);\fi\fi
  }

  \foreach \x in {0,...,28} {
    \pgfmathsetmacro\y{mod(\x*\x,29)}
    \pgfmathsetmacro\xnew{28 - \y}
    \pgfmathsetmacro\ynew{\x + 3}
    \ifdim\xnew pt>10pt \filldraw[orange!90!black] (\xnew,\ynew) circle (0.3);
    \else\ifdim\ynew pt>10pt \filldraw[orange!90!black] (\xnew,\ynew) circle (0.3);\fi\fi
  }
\end{tikzpicture}
\caption{Shifting up and deleting the points lying in the left-bottom square.}
\label{fig:shift}
\end{figure}
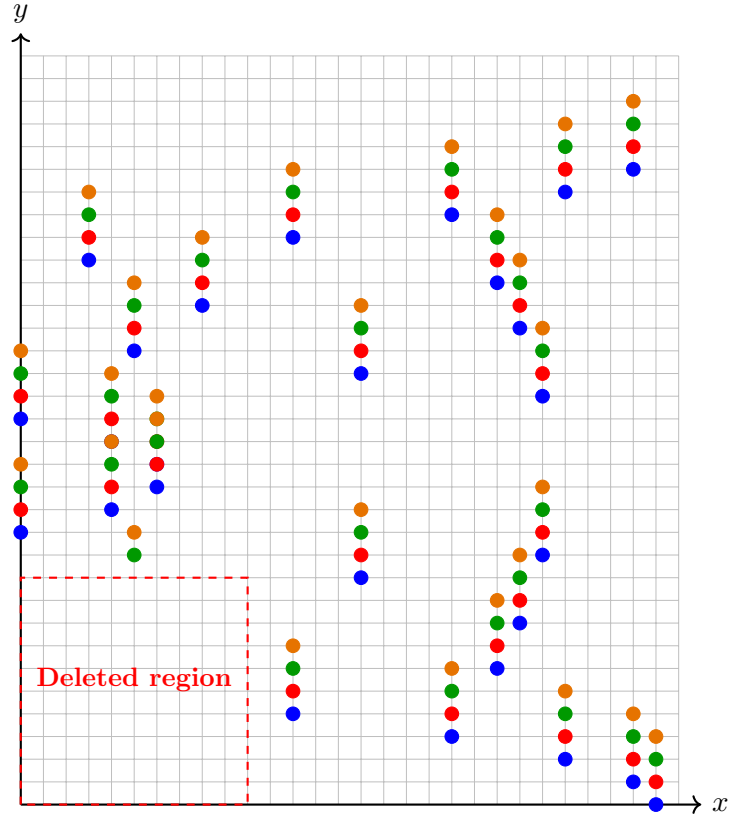

Let $h_n = |L \cap [1,n]^2|$ (for all $n\in \nn^+$) denote the number of points of the no-3-in-line construction $L$ lying in $[1,n]^2$.

Choosing $x_n = \left\lfloor \frac{k(n)}{2} \right\rfloor - \left\lfloor \frac{k(n-1)}{2} \right\rfloor$ (and $k(0)=0$) guaranties $\sum_{i=1}^{n} 2x_i \leq k(n)$. As $L_{i,j}\cap [0, i-1]^2 = \emptyset$, this ensures that the copies of $L$ contain at most $\sum _{i=1}^{n} x_i\le k(n)$ points in $[1,n]^2$.

\medskip
\noindent Now we will estimate the number of points of $S$ within $[1, n +k(n)/2]^2$. By the choice of the sequence $(x_n)_{n\in\nn}$, all points of 
$$S_n:=\bigcup_{\substack{1 \le i\le n  \\ 0 \le j < x_i} }L_{i,j}$$

\noindent lie within $[1, n+k(n)/2]^2$. Therefore, $|S_n|$ is a lower bound for the number of points in $[1, n+k(n)/2]^2$.

$L$ contains at most two points on each vertical line, therefore, there is no point in $S$ that is contained by more than two distinct $L_{i,j}$ copies of $L$.  
Consequently,
$$\left | \bigcup_{\substack{1\le i \le n \\0\le j < x_i}}L_{i,j}\cap [1,n+k(n)/2]^2\right | \ge \frac{1}{2} \sum _{i = 1}^n \sum _{j = 0}^{x_i-1} |L_{i,j}\cap [1,n+k(n)/2]^2|.$$

We acquired $L_{i,j}$ by removing the points in $[1, i-1]^2 from  L+v_{i,j}$. The point set $(L+v_{i,j})\cap [1, i-1]^2$ is a subset of $L\cap [1,i-1]^2$, thus it contains at most $h_{i-1}$ points. Furthermore, $[1, n+k(n)/2]^2$ contains all points of $L_{i,j}$. Therefore
$$\left|L_{i,j}\cap [1,n+k(n)/2]^2\right|=\left|(L+v_{i,j})\cap [1,n+k(n)/2]^2 \setminus (L+v_{i,j}) \cap[1, i-1]^2\right| \ge h_n-h_{i-1}.$$

\noindent The number of points of $S$ in $[n+k(n)/2]^2$ is at least
$$
\frac{1}{2} \sum _{i = 1}^n \sum _{j = 0}^{x_i-1} \left|L_{i,j}\cap [1,n+k(n)/2]^2\right|\ge \frac{1}{2}\sum_{i=1}^{n} x_i (h_n - h_{i-1}) = 
\frac{1}{2}\sum_{i=1}^{n} \left(\left\lfloor \frac{k(i)}{2} \right\rfloor - \left\lfloor \frac{k(i-1)}{2} \right\rfloor\right) (h_n - h_{i-1})
$$
Expanding the telescopic sum, we get:
$$ \frac{1}{2}\sum_{i=1}^{n} x_i (h_n - h_{i-1}) = \frac{1}{2} h_n \left(\left\lfloor \frac{k(n)}{2} \right\rfloor - \left\lfloor \frac{k(0)}{2} \right\rfloor \right) - \frac{1}{2}\sum_{i=1}^{n} h_{i-1} \left( \left\lfloor \frac{k(i)}{2} \right\rfloor - \left\lfloor \frac{k(i-1)}{2} \right\rfloor \right) 
$$
Grouping by $\lfloor k(\cdot) \rfloor$ and noting $k(0)=0$: 
$$
\left|S\cap [1, n+k(n)/2]^2\right| \ge \frac{1}{2}\sum_{i=1}^{n} x_i (h_n - h_{i-1}) = \frac{1}{2}\sum_{i=1}^{n} \left\lfloor \frac{k(i)}{2} \right\rfloor (h_i - h_{i-1})
$$

For each $m$ with $n+k(n)/2\le m < (n+1)+k(n+1)/2$, we have
    
$$  \displaystyle\liminf_{m\to\infty} \frac{\left|S \cap [1,m]^2\right|}{m k(m)}\ge \displaystyle\liminf_{n\to\infty} \frac{\left|S \cap [1,n +k(n)/2]^2\right|}{\left((n+1) + k(n+1)/2\right)k\left((n+1) + k(n+1)/2\right)}=$$

$$
=\displaystyle\liminf_{n\to\infty}\frac{\displaystyle\frac{1}{2}\sum_{i=1}^{n} \left\lfloor \frac{k(i)}{2} \right\rfloor (h_i - h_{i-1})}{\left((n+1) + k(n+1)/2\right)k\left((n+1) + k(n+1)/2\right)}\
$$

Note that this construction led to a $\liminf$ that depends on the $h_i$ sequence. The next goal is to find the worst-case $h: \Z^+ \to \Z^+, i \mapsto h_i$ monotonous function so that we can provide a lower bound on the density.

Since the $n$th numerator of the $\liminf$ uses the first $n+1$ values of $h$, we will investigate these $h |_{[0,n]}$ restrictions in depth. First of all, $h(x) \le 2 x$ for every  $x \in [0, n]$. Second, $\liminf_{n \to \infty} \frac{h(n)}{2n} = H$ corresponds to the fact that $\forall \varepsilon > 0$ there exists a large enough $n_0$ such that $\forall n \ge x > n_0: (1 - \varepsilon) 2Hx \le h|_{[0,n]} (x)$.

Our goal is to provide a lower bound to the $\liminf$ by introducing a family of $h^{(n)}: \{1,...n\} \to \Z^+$ functions, which generalize the idea of restrictions and minimize the inner sequence of $\liminf$ elements, element by element.

\begin{align*}
 \displaystyle\liminf_{m\to\infty} \frac{\left|S \cap [1,m]^2\right|}{m k(m)} &\geq \inf_{h: \Z^+ \to \Z^+ } \left( \displaystyle\liminf_{n\to\infty} \frac{\displaystyle\frac{1}{2}\sum_{i=1}^{n} \left\lfloor \frac{k(i)}{2} \right\rfloor (h_i - h_{i-1}) }{{(n+1+k(n+1)/2)\cdot k(n+1+k(n+1)/2)}}\right) \geq
\\&\geq  \displaystyle\liminf_{n\to\infty} \left( \frac{\displaystyle\inf_{h^{(n)}: \{ 1 \dots n \} \to \Z^+ } \displaystyle\frac{1}{2}\sum_{i=1}^{n} \left\lfloor \frac{k(i)}{2} \right\rfloor (h^{(n)} (i) - h^{(n)} (i-1))}{{(n+1+k(n+1)/2)\cdot k(n+1+k(n+1)/2)}} \right)
\end{align*}
So, one should find such $h^{(n)}: \{1,...n\} \to \Z^+; (1 - \varepsilon) 2 H x \leq h^{(n)} (x) \leq 2x$ functions that 
\begin{equation} \label{eq:wsum}
    \frac{1}{2}\sum_{i=1}^{n} \left\lfloor \frac{k(i)}{2} \right\rfloor \left(h^{(n)}(i) - h^{(n)}(i-1) \right)
\end{equation}
is minimal. By the monotonicity of $k$, $\left\lfloor \frac{k(i)}{2} \right\rfloor$ is also monotonous. To minimize the value of (\ref{eq:wsum}), choose a function $h^{(n)}$ with $\left(h^{(n)}(i) - h^{(n)}(i-1) \right) = 0$ for large $i$ (this is a slightly modified version of the Rearrangement Inequality).  This corresponds to $h^{(n)}(x) = 2x$ if $x < (1 - \varepsilon) Hn$; otherwise, $h^{(n)}(x) = (1 - \varepsilon) 2Hn$.

\noindent The value of the sum is:
$$
\frac{1}{2}\sum_{i=1}^{(1 - \varepsilon)Hn} 2 \left\lfloor \frac{k(i)}{2} \right\rfloor \geq \sum_{i=1}^{(1 - \varepsilon)Hn} \left(\frac{k(i)}{2} - \frac{1}{2}\right)= -\frac{(1 - \varepsilon)Hn}{2} + \frac{1}{2}\sum_{i=1}^{(1 - \varepsilon)Hn} k(i)
$$

\noindent By definition \ref{sublin3} of the sublinear limiter, $-(1 - \varepsilon)Hn/2$ is much smaller than the denominator $(n+1+k(n+1)/2)k(n+1+k(n+1)/2)$ in the limit. Therefore,

\begin{align*}
\displaystyle\liminf_{m\to\infty} \frac{\left|S \cap [1,m]^2\right|}{m k(m)} &\ge
\displaystyle\liminf_{n\to\infty}  \frac{\frac{1}{2}\sum_{i=1}^{(1 - \varepsilon)Hn} k(i)}{(n+1+k(n+1)/2)k(n+1+k(n+1)/2)} \overset{\ref{sublin1}}{\ge}
\\&\geq \displaystyle\liminf_{n\to\infty} \frac{(1 - \varepsilon)Hn k(\lfloor(1 - \varepsilon)Hn\rfloor)}{4(n+1+k(n+1)/2)k(n+1+k(n+1)/2)} =
\\&= \displaystyle\liminf_{n\to\infty} \frac{(1 - \varepsilon)H}{4} \frac{k(\lfloor(1 - \varepsilon)Hn\rfloor)}{k(n)}\ge  \frac{\left((1 - \varepsilon)H\right)^2}{4}
\end{align*}

We used Property \ref{sublin1} in the definition of the sublinear limiter to provide a lower bound for the numerator in the second inequality. The last equality holds in the limit by Property \ref{sublin5}. Finally, $\varepsilon \to 0$ when $n\to \infty$ implies
$$\liminf_{m\to\infty} \frac{\left|S \cap [1,m]^2\right|}{m k(m)}\ge \frac{H^2}{4}.$$

\end{proof}

\begin{corollary}

$H^2/4$ is a universal lower bound for $\liminf \frac{|S_n|} {nk(n)}$ if $k(n)=n^{\alpha}$. If $k(n) = \log (n)$, then
$$\liminf_{n\to\infty} \frac{H}{4} \frac{k(\lfloor Hn\rfloor)}{k(n)}\ge\frac{H}{4}$$
is a universal lower bound.

\end{corollary}

{
    \footnotesize
    
    \bibliographystyle{abbrv}
	\bibliography{main}
    
}\end{document}